\documentclass[12pt]{amsart}
\usepackage[utf8]{inputenc}

\usepackage[margin=1.3in]{geometry}

\usepackage{amsmath}
\usepackage{amsfonts}
\usepackage{amssymb}
\usepackage{amsthm}
\usepackage{mathtools}
\usepackage{caption}
\usepackage{subcaption}
\usepackage{bbm}
\usepackage[export]{adjustbox}

\usepackage{stmaryrd}

\usepackage[all]{xy}

\usepackage{tikz-cd}
\usetikzlibrary{matrix}
\usepackage{graphicx} 
\usepackage{float}

\usepackage{epstopdf}

\usepackage[linktocpage]{hyperref}
\hypersetup{
    colorlinks=true,
    linkcolor=blue,
    citecolor=blue,      
    urlcolor=blue,
}

\usepackage{color}
\definecolor{note}{rgb}{0,0,1}  

\newtheorem{theorem}{Theorem}
\newtheorem{definition}[theorem]{Definition}
\newtheorem{proposition}[theorem]{Proposition}
\newtheorem{lemma}[theorem]{Lemma}

\newtheorem{corollary}[theorem]{Corollary}

\newtheorem{remark}[theorem]{Remark}

\numberwithin{equation}{section}
\numberwithin{theorem}{section}

\usepackage{enumitem}

\usepackage{todonotes}

\newcommand{\op}{\operatorname}

\newcommand{\Z}{\mathbb{Z}}
\newcommand{\Q}{\mathbb{Q}}

\newcommand{\be}{\begin{enumerate}}
\newcommand{\ee}{\end{enumerate}}

\usepackage[english]{babel}

\usepackage{hyphenat}

\usepackage[backend=bibtex,style=alphabetic,maxalphanames=4,maxnames=4]{biblatex}

\renewbibmacro{in:}{}

\DeclareDelimFormat[bib,biblist]{nametitledelim}{\addcomma\space}

\DeclareFieldFormat*{title}{\mkbibitalic{#1}\addcomma}
\DeclareFieldFormat*{journaltitle}{#1}
\DeclareFieldFormat*{volume}{\mkbibbold{#1}}
\DeclareFieldFormat{pages}{#1}
\DeclareFieldFormat[misc]{date}{preprint {#1}}
\DeclareFieldFormat{mr}{%
  MR\addcolon\space
  \ifhyperref
    {\href{http://www.ams.org/mathscinet-getitem?mr=MR#1}{\nolinkurl{#1}}}
    {\nolinkurl{#1}}}
    
\AtEveryBibitem{
  \clearfield{url}
  \clearfield{number}
  \clearfield{doi}
  \clearfield{issn}
  \clearfield{isbn}
  \clearfield{eprintclass}
}
\AtEveryBibitem{\ifentrytype{book}{\clearfield{pages}}{}}

\bibliography{hecke}

\usepackage{fancyhdr}
\usepackage{todonotes}

\title{Higher-dimensional Heegaard Floer homology and the polynomial representation of double affine Hecke algebras}

\author{Yuan Gao}
\address{School of Mathematics, Nanjing University, Nanjing 210093, China}
\email{yuangao@nju.edu.cn} \urladdr{}

\author{Eilon Reisin-Tzur}
\address{University of California, Los Angeles, Los Angeles, CA 90095}
\email{eilon.tzur@gmail.com} \urladdr{}

\author{Yin Tian}
\address{School of Mathematical Sciences, Beijing Normal University; 
Laboratory of Mathematics and Complex Systems, Ministry of Education, Beijing 100875, China}
\email{yintian@bnu.edu.cn} \urladdr{}

\author{Tianyu Yuan}
\address{School of Mathematical Sciences, Eastern Institute of Technology, Ningbo, Zhejiang, 315200, China}
\email{tyyuan@eitech.edu.cn} \urladdr{}

\date{\today}

\subjclass[2010]{Primary 53D40; Secondary 20C08.}

\begin{document}

\maketitle

\begin{abstract}
    We show that the higher-dimensional Heegaard Floer homology between tuples of cotangent fibers and the conormal bundle of a homotopically nontrivial simple closed curve on $T^2$ recovers the polynomial representation of double affine Hecke algebra of type A. 
    We also give a topological interpretation of Cherednik's inner product on the polynomial representation.
\end{abstract}

\tableofcontents

\section{Introduction}


In \cite{honda2022jems,tian2025}, Honda, Tian and Yuan showed that the higher-dimensional Heegaard Floer (HDHF) group of tuples of cotangent fibers can recover various kinds of Hecke algebras of type A.
The second author then showed that the HDHF group of tuples of conormal bundles of disjoint simple closed curves on a surface gives a module structure over the surface Hecke algebras \cite{reisintzur2023floertheoretic}. In particular, it can recover the polynomial representation of the double affine Hecke algebra $\ddot{H}_{\kappa}$ when the surface is a torus.
However, \cite{reisintzur2023floertheoretic} considers tuples of conormal bundles of parallel copies of closed curves. The number of generators is larger than that for the polynomial representation, and one has to average over permutation elements to get the expected result.

In this paper, we provide a more natural approach to the polynomial representation of $\ddot{H}_{\kappa}$ via HDHF of the conormal bundle of a single closed curve $\alpha$.
More interesting phenomena occur since pseudoholomorphic curves may undergo self-intersections along $\alpha$.
This leads to a new type of symplectic skein relations compared to that of Ekholm and Shende \cite{ekholm2021skeins}; see Section \ref{section-local}.

Let $\Sigma$ be a closed surface with $\mathrm{genus}\geq1$, $\alpha$ be a nontrivial simple loop on $\Sigma$, and $\mathbf{q}=\{q_1,\dots,q_{\kappa}\}$ be a $\kappa$-tuple of distinct points on $\Sigma$ away from $\alpha$. 
On the skein theory side, Morton and Samuelson  \cite{morton2021dahas} defined the {\em braid skein algebra} $\op{BSk}_\kappa(\Sigma,\mathbf{q})$, which is a quotient of the braid group of $\Sigma$ by certain skein relations.
We consider a variant, the {\em braid skein module} $\op{BSk}_\kappa(\Sigma,\mathbf{q},\alpha)$ which is generated by braids from $\mathbf{q}$ to $\alpha$.  
It is naturally a module over $\op{BSk}_\kappa(\Sigma,\mathbf{q})$ via the concatenation of braids along $\mathbf{q}$.
On the symplectic geometry side, we define the wrapped HDHF group between the conormal bundle $T^*_\alpha\Sigma$ and the tuples of cotangent fibers $T^*_\mathbf{q}\Sigma$, denoted by $HW(T^*\Sigma,\mathbf{q},\alpha,\mathbf{p}_0)^\varphi$, where the other parameters are to be explained later.
This HDHF group is naturally a module over $HW(T^*\Sigma,\mathbf{q})$, the HDHF endomorphism algebra of the cotangent fibers $T^*_\mathbf{q}\Sigma$.
This algebra is  isomorphic to $\op{BSk}_\kappa(\Sigma,\mathbf{q})\otimes_{\mathbb{Z}[\hbar]}\mathbb{Z}\llbracket\hbar\rrbracket$; see \cite{honda2022jems}.
Our first result is the following. 

\begin{theorem} \label{thm main1}
There is an isomorphism 
$$HW(T^*\Sigma,\mathbf{q},\alpha,\mathbf{p}_0)^\varphi \cong \op{BSk}_\kappa(\Sigma,\mathbf{q},\alpha)\otimes_{\mathbb{Z}[\hbar]}\mathbb{Z}\llbracket\hbar\rrbracket,$$ which commutes with the actions of $HW(T^*\Sigma,\mathbf{q})$ and $\op{BSk}_\kappa(\Sigma,\mathbf{q})$, respectively.
\end{theorem}

We then focus on the case $\Sigma=T^2$. 
Morton and Samuelson showed that the braid skein algebra $\op{BSk}_\kappa(T^2,\mathbf{q})$ is isomorphic to the double affine Hecke algebra $\ddot{H}_{\kappa}$ \cite{morton2021dahas}.
We prove that the braid skein module $\op{BSk}_\kappa(\Sigma,\mathbf{q},\alpha)$ is isomorphic to the polynomial representation $P_{\kappa}$ of $\ddot{H}_{\kappa}$. 
So we have the following corollary.
\begin{corollary} \label{cor}
   The HDHF group $HW(T^*T^2,\mathbf{q},\alpha,\mathbf{p}_0)^\varphi$ is isomorphic to the polynomial representation of $\ddot{H}_{\kappa}$ after tensoring with $\mathbb{Z}\llbracket\hbar\rrbracket$.
\end{corollary}

The double affine Hecke algebra is defined by Cherednik to study the Macdonald polynomials. 
The nonsymmetric Macdonald polynomials give a basis of $P_{\kappa}$
with some nice properties. 
In particular, Cherednik constructed an inner product on $P_{\kappa}$ so that the Macdonald polynomials are pairwise orthogonal. 
Moreover, the action of $\ddot{H}_{\kappa}$ on $P_{\kappa}$ has some adjoint property with respect to Cherednik's inner product.

Our second result is to give a topological interpretation of Cherednik's inner product via braid skeins on $T^2$. 
Consider a composition of two maps: 
$$\op{BSk}_\kappa(T^2,\mathbf{q},\alpha) \otimes \op{BSk}_\kappa(T^2,\mathbf{q},\alpha) \to \op{BSk}_\kappa(T^2,\alpha,\mathbf{q}) \otimes \op{BSk}_\kappa(T^2,\mathbf{q},\alpha) \to \op{BSk}_\kappa(T^2,\alpha,\alpha),$$
where the first map is taking the inverse braid for the first factor, and the second map is a concatenation of braids along $\mathbf{q}$. 
We show that the image of the map is one dimensional over the ground field. 
So we have a bilinear form on $\op{BSk}_\kappa(T^2,\mathbf{q},\alpha)$.

\begin{theorem} \label{thm main2}
This bilinear form is isomorphic to Cherednik's inner product under the isomorphism $\op{BSk}_\kappa(T^2,\mathbf{q},\alpha) \cong P_{\kappa}$.
\end{theorem}

The idea of proof is to use the fact that Cherednik's inner product is uniquely determined by the adjoint property for the action of $\ddot{H}_{\kappa}$; see \cite[Proposition 3.3.2]{Cherednik}.  
By the topological construction our bilinear form has the same adjoint property for the action of $\op{BSk}_\kappa(T^2,\mathbf{q})$ which is isomorphic to $\ddot{H}_{\kappa}$. 

\vspace{.2in}
\noindent {\em Further directions:} 
\begin{itemize}
\item It is natural to ask whether we can realize Cherednik's inner product in Floer theory. The difficulty here is to define an appropriate HDHF group between the conormal bundle $T_{\alpha}^*T^2$ to itself since $T_{\alpha}^*T^2$ is a singular Lagrangian in HDHF setting. 
\item On the topological side, it is interesting to look for the braid skein representatives in $\op{BSk}_\kappa(T^2,\mathbf{q},\alpha)$ corresponding to the nonsymmetric Macdonald polynomials. We can then explore the interesting combinatorics via topology on $T^2$.
\end{itemize}




\noindent \textit{Acknowledgements}. 
YT is supported by NSFC 12471064. TY is supported by NSFC 12501080.
The authors thank Ko Honda for helpful discussions.

\section{Braid skein modules}

We review the basics about braid skein algebras and then discuss braid skein modules in this section.

\subsection{Braid skein algebras}
Recall the definition of the braid skein algebra associated to a given surface $\Sigma$ due to Morton and Samuelson \cite{morton2021dahas}.

Let $\mathrm{UConf}_{\kappa}(\Sigma)=\{\{q_1,\dots,q_{\kappa}\}~|~q_i\in \Sigma, ~q_i\neq q_j ~\mbox{for}~i\neq j\}$ be the unordered configuration space of $\kappa$ disjoint points on $\Sigma$. 
Fix a basepoint $\mathbf{q} \in \mathrm{UConf}_{\kappa}(\Sigma)$.
The based loop space $\Omega(\mathrm{UConf}_{\kappa}(\Sigma),\mathbf{q})$ consists of $\kappa$-strand braids on $\Sigma$.
Note that $H_0(\Omega(\mathrm{UConf}_{\kappa}(\Sigma),\mathbf{q}))$ is isomorphic to the group algebra $\mathbb{Z}[\mathrm{Br}_{\kappa}(\Sigma,\mathbf{q})]$, where $\mathrm{Br}_{\kappa}(\Sigma,\mathbf{q})$ denotes the braid group of $\Sigma$ based at $\mathbf{q}$.
Given two basepoints $\mathbf{q},\mathbf{q}'\in\mathrm{UConf}_{\kappa}(\Sigma)$, we also write $\mathrm{Br}_{\kappa}(\Sigma,\mathbf{q},\mathbf{q}')$ to denote the space of braids that start from $\mathbf{q}$ and end on $\mathbf{q}'$.

Fix a marked point $\star \in \Sigma$ which is disjoint from $\mathbf{q}$. 
Let $\mathrm{Br}_{\kappa,1}(\Sigma,\mathbf{q},\star)$ be the subgroup of $\mathrm{Br}_{\kappa+1}(\Sigma,\mathbf{q}\sqcup\{\star\})$ consisting of braids whose last strand connects $\star$ to itself by a straight line in $[0,1]\times\Sigma$.
For notation simplicity, when we write $\gamma\in\mathrm{Br}_{\kappa,1}(\Sigma,\mathbf{q},\star)$, we always mean the first $\kappa$ strands.

\begin{definition}[Morton-Samuelson]
    \label{def-skein}
    The braid skein algebra $\mathrm{BSk}_\kappa(\Sigma,\mathbf{q})$ is the quotient of the group algebra $\mathbb{Z}[s^{\pm1}, c^{\pm1}][\mathrm{Br}_{\kappa,1}(\Sigma,\mathbf{q},\star)]$ by two local relations:
    \begin{enumerate}
    \item the HOMFLY skein relation 
        \begin{equation}
            \label{eq-skein'}
            \includegraphics[width=1cm,valign=c]{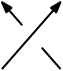}-\includegraphics[width=1cm,valign=c]{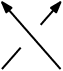}=(s-s^{-1})\includegraphics[width=1cm,valign=c]{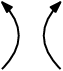},
        \end{equation}
    \item the marked point relation
        \begin{equation}
            \label{eq-c}
            \includegraphics[height=2cm,valign=c]{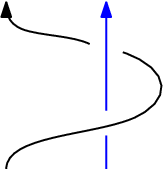}\,=\,c^2\,\includegraphics[height=2cm,valign=c]{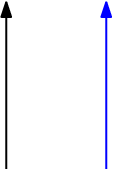}.
        \end{equation}
    \end{enumerate}
    Here the black lines are strands between basepoints in $\mathbf{q}$ and the straight blue line connects the marked point $\star$ to itself. 
    
    The product is given by the concatenation of braids. 
\end{definition}

In this paper, we denote $\hbar=s-s^{-1}$.

\subsection{Braid skein modules}
Let $\alpha$ be a homotopically nontrivial simple closed curve on $\Sigma$ which is disjoint from $\mathbf{q}\sqcup\{\star\}$.
Consider the path space
\begin{equation*}
    \Omega(\op{UConf}_\kappa(\Sigma),\mathbf{q},\alpha)\coloneqq \{\gamma\in C^0([0,1],\op{UConf}_\kappa(\Sigma))\,|\,\gamma(0)=\mathbf{q},\gamma(1)\subset \alpha\}.
\end{equation*}

Denote the set of homotopy classes of  paths in $\Omega(\op{UConf}_\kappa(\Sigma),\mathbf{q},\alpha)$, viewed as braids that start from $\mathbf{q}$ and end on $\alpha$, by $\op{Br}_\kappa(\Sigma,\mathbf{q},\alpha)$.
Similarly, we define $\op{Br}_{\kappa,1}(\Sigma,\mathbf{q},\alpha,\star)$ to consist of braids that start from $\mathbf{\mathbf{q}\sqcup\{\star\}}$ and end on $\alpha\sqcup\{\star\}$ such that the last strand connects $\star$ to itself by a straight line in $[0,1]\times\Sigma$.

\begin{definition}
    \label{def-skein-module}
    The braid skein module $\mathrm{BSk}_\kappa(\Sigma,\mathbf{q},\alpha)$ is the quotient of the free module $\mathbb{Z}[s^{\pm1}, c^{\pm1}][\mathrm{Br}_{\kappa,1}(\Sigma,\mathbf{q},\alpha,\star)]$ by the same local relations in Definition \ref{def-skein} together with an additional relation:
    \begin{enumerate}
    
    \item[(3)] the skein relation at a corner
        \begin{equation}
            \label{eq-corner}
            \includegraphics[height=2cm,valign=c]{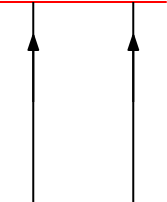}\,=\,s\,\includegraphics[height=2cm,valign=c]{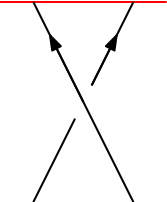}\,=\,s^{-1}\,\includegraphics[height=2cm,valign=c]{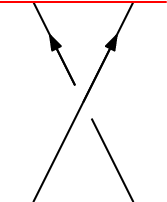},
        \end{equation}
    \end{enumerate}
    where the red arcs denote $\alpha$. 
    
    The module structure over $\mathrm{BSk}_\kappa(\Sigma,\mathbf{q})$ is given by the concatenation of braids at $\mathbf{q}$. 
\end{definition}


\begin{remark}
       The skein relation at a corner is equivalent to the ends-slide relation defined in \cite{reisintzur2023floertheoretic} in the case where the $\alpha_i$ are parallel copies of $\alpha$. To realize this equivalence, we choose conventions to translate between the setup of this paper and that of \cite{reisintzur2023floertheoretic}. We translate the image on the left side of \ref{eq-corner} to the same braid but with the strand on the right ending at a push-off of $\alpha$ which is farther away. Sliding the longer strand on the right past the shorter strand to its left results in a crossing where the shorter strand is on top since the longer strand viewed as a time-1 Hamiltonian chord must travel farther. The ends-slide relation gives us a factor of $s^{-1}$ (where we let $d=s$) attached to the resulting braid. Identifying the push-off of $\alpha$ with $\alpha$, we get the image on the right side of \ref{eq-corner}. 
       
\end{remark}

For later use, we consider the braid skein module on $T^*_\alpha\Sigma$, the conormal bundle of $\alpha$, which is topologically an annulus.
Let $\mathbf{p}_0 \in \mathrm{UConf}_{\kappa}(T^*_\alpha\Sigma\setminus\alpha)$.
The braid groups $\mathrm{Br}_{\kappa}(T^*_\alpha\Sigma,\mathbf{p}_0)$ and $\mathrm{Br}_{\kappa}(T^*_\alpha\Sigma,\mathbf{p}_0,\alpha)$ are defined in the similar way as above.
We then define $\mathrm{BSk}^\circ_\kappa(T^*_\alpha\Sigma,\mathbf{p}_0)$ as the quotient of the free module $\mathbb{Z}[s^{\pm1}][\mathrm{Br}_{\kappa}(T^*_\alpha\Sigma,\mathbf{p}_0)]$ by the local relation (\ref{eq-skein'}) and define $\mathrm{BSk}^\circ_\kappa(T^*_\alpha\Sigma,\mathbf{p}_0,\alpha)$ as the quotient of the free module $\mathbb{Z}[s^{\pm1}][\mathrm{Br}_{\kappa}(T^*_\alpha\Sigma,\mathbf{p}_0,\alpha)]$ by the local relations (\ref{eq-skein'}) and (\ref{eq-corner}).
Here the superscript $\circ$ indicates that we omit the $c$-parameter and the relation (\ref{eq-c}).
Given $\mathbf{p}_0,\mathbf{p}_1\in \mathrm{UConf}_{\kappa}(T^*_\alpha\Sigma\setminus\alpha)$, we define $\mathrm{BSk}^\circ_\kappa(T^*_\alpha\Sigma,\mathbf{p}_0,\mathbf{p}_1)$ analogously.


Consider the space of pairs of paths
\begin{equation*}
    \Omega(\mathbf{q},\alpha,\mathbf{p}_0)\coloneqq\{(\gamma,\eta)\in \Omega(\op{UConf}_\kappa(T^*_\alpha\Sigma),\mathbf{p}_0,\alpha)\times\Omega(\op{UConf}_\kappa(\Sigma),\mathbf{q},\alpha)\,|\,\gamma(1)=\eta(1)\}.
\end{equation*}
Let $\op{Br}_\kappa(\mathbf{q},\alpha,\mathbf{p}_0)$ be the set of homotopy classes of pairs of paths in $\Omega(\mathbf{q},\alpha,\mathbf{p}_0)$.
We define $\op{Br}_{\kappa,1}(\mathbf{q},\alpha,\mathbf{p}_0,\star)$ in a similar manner to $\mathrm{Br}_{\kappa,1}(\Sigma,\mathbf{q},\alpha,\star)$.

\begin{definition}
    \label{def-skein-module-cotangent}
    The braid skein module $\mathrm{BSk}_\kappa(\mathbf{q},\alpha,\mathbf{p}_0)$ is the quotient of the free module $\mathbb{Z}[s^{\pm1},c^{\pm1}][\op{Br}_{\kappa,1}(\mathbf{q},\alpha,\mathbf{p}_0,\star)]$ by the local relations for $\mathrm{BSk}^\circ_\kappa(T^*_\alpha\Sigma,\mathbf{p}_0,\alpha)$ and $\mathrm{BSk}_\kappa(\Sigma,\mathbf{q},\alpha)$ on their respective components of $\Omega(\mathbf{q},\alpha,\mathbf{p}_0)$.
\end{definition}

\section{The Floer interpretation of braid skein modules}

We give a HDHF interpretation of the braid skein modules in this section. 
More precisely, we realize $\mathrm{BSk}_\kappa(\Sigma,\mathbf{q},\alpha,\mathbf{p}_{0})$ as the HDHF group
between the conormal bundle $T^*_\alpha\Sigma$ and the tuples of cotangent fibers $T^*_\mathbf{q}\Sigma$.

\subsection{The Floer homology}
In this subsection, we construct the HDHF chain complex $CW(T^*\Sigma,\mathbf{q},\alpha,\mathbf{p}_0)$, and endow it with the structure of a left $CW(T_{\mathbf{q}}^*\Sigma)$-module.
Here $T_{\mathbf{q}}^*\Sigma\coloneqq \sqcup_{q_i\in\mathbf{q}}T^*_{q_i}\Sigma$ and $CW(T_{\mathbf{q}}^*\Sigma)$ is the HDHF chain complex of the $\kappa$-tuple of cotangent fibers $T_{\mathbf{q}}^*\Sigma$.
We refer the reader to \cite{colin2020applications} or \cite[Section 2.1]{honda2022jems} for the detailed definition of HDHF, and we omit most details here. 

Given two $\kappa$-tuples of disjoint exact Lagrangians $L_i=L_{i1}\sqcup\dots\sqcup L_{i\kappa}\subset T^*\Sigma$, $i=0,1$, whose components are mutually transverse, let $CF(L_0,L_1)$ be the free abelian group generated by all $\boldsymbol{y}=\{y_{1},\dots,y_{\kappa}\}$ where $y_{j}\in L_{0j}\cap L_{1\sigma(j)}$ and $\sigma$ is some permutation of $\{1,\dots,\kappa\}$. 
The coefficient ring is set to be $\mathbb{Z}\llbracket\hbar\rrbracket$. 

Let $g$ be a Riemannian metric on $\Sigma$ and $|\cdot|$ be the induced norm on $T^*\Sigma$. Choose a time-dependent ``quadratic at infinity" Hamiltonian
\begin{gather}
    \label{eq-H}
    H_V\colon [0,1]\times T^*{M}\to\mathbb{R},\quad H_V(t,q,p)=\frac{1}{2}|p|^2+V(t,q),
\end{gather}
where $t\in[0,1]$, $q\in \Sigma$, $p\in T^*_q \Sigma$, and $V$ is some perturbation term with small $W^{1,2}$-norm. The Hamiltonian vector field $X_{H_V}$ with respect to the canonical symplectic form $\omega=dq\wedge dp$ is then given by $i_{X_{H_V}}\omega=dH_V$. Let $\phi^t_{H_V}$ be the time-$t$ flow of $X_{H_V}$. 

By choosing $g$ and $V$ generically, we can guarantee that all Hamiltonian chords of $\phi^t_{H_V}$ between the cotangent fibers $\{T_{q_1}^*{\Sigma},\dots,T_{q_\kappa}^*{\Sigma}\}$ are nondegenerate and of degree 0. 
In each homotopy class, we arrange $\alpha$ such that all Hamiltonian chords of $\phi^t_{H_V}$ from the cotangent fibers $\{T_{q_1}^*{\Sigma},\dots,T_{q_\kappa}^*{\Sigma}\}$ to $T^*_\alpha\Sigma$ are of degree 0.
This is always possible when the genus of $\Sigma$ is greater or equal to 1.
We refer to \cite[Definition 3.1, Proposition 3.3]{abbondandolo2008homology} for details of the gradings.

We define $CW(T_{\mathbf{q}}^*\Sigma,T_{\alpha}^*\Sigma)\coloneqq CF(\phi^1_{H_V}(T_{\mathbf{q}}^*\Sigma),T_{\alpha}^*\Sigma)\otimes \mathbb{Z}[c^{\pm1}]$, i.e., the generators are $\kappa$-tuples of time-1 Hamiltonian chords that start from $T_\mathbf{q}^*\Sigma$ and end on $T_\alpha^*\Sigma$.
Similarly, we define the wrapped HDHF complex $CW(T_{\mathbf{q}}^*\Sigma)\coloneqq CF(\phi^1_{H_V}(T_{\mathbf{q}}^*\Sigma),T_{\mathbf{q}}^*\Sigma)\otimes \mathbb{Z}[c^{\pm1}]$.

Let $D$ be the unit disk in $\mathbb{C}$ and $D_{m}=D-\{w_0,\dots,w_m\}$, where $w_i\in\partial D$ are boundary marked points arranged counterclockwise. Let $\partial_i D_{m}$ be the boundary arc from $w_i$ to $w_{i+1}$. 
Let $\mathcal{A}_m$ be the moduli space of $D_{m}$ modulo automorphisms; we choose representatives $D_{m}$ of equivalence classes of $\mathcal{A}_m$ in a smooth manner and abuse notation by writing $D_{m}\in \mathcal{A}_m$.
We call $D_{m}$ the ``$A_\infty$ base direction''.

\begin{figure}[ht]
    \centering
    \includegraphics[width=5cm]{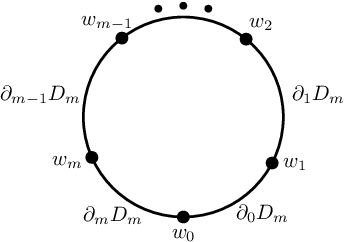}
    \caption{The $A_\infty$ base direction $D_{m}$.}
    \label{fig-base}
\end{figure}

The full ambient symplectic manifold is $(D_{m}\times T^*\Sigma,\,\Omega_{m}=\omega_{m}+\omega)$, where $\omega_{m}$ is an area form on $D_{m}$ which restricts to $ds_i\wedge dt_i$ on the strip-like ends $e_i\simeq [0,\infty)_{s_i}\times[0,1]_{t_i}$ around $p_i$ for $i=1,\dots,m$ and $e_i\simeq (-\infty,0]_{s_i}\times[0,1]_{t_i}$ around $p_0$. 
Let $\pi_{T^*\Sigma}\colon D_{m}\times T^*\Sigma\to T^*\Sigma$ be the projection to $T^*\Sigma$ and 
$\pi_{D_{m}}$ be the symplectic fibration $\pi_{D_{m}}\colon (D_{m}\times T^*\Sigma,\Omega_{m})\to (D_{m},\omega_{m})$.

We write $T^*\Sigma=DT^*\Sigma\cup [0,\infty)_s\times \partial DT^*\Sigma$ where $DT^*\Sigma$ is the unit cotangent bundle of $\Sigma$ with respect to the metric $|\cdot|$.
Let $\mathcal{J}_{T^*\Sigma}$ be the set of $\omega$-compatible almost complex structures $J_{T^*\Sigma}$ on $(T^*\Sigma,\omega)$ that are asymptotic to an almost complex structure on the symplectization end $[0,\infty)_s\times \partial DT^*\Sigma$ that takes $\partial_s$ to the Reeb vector field of $qdp|_{\partial DT^*\Sigma}$, takes $\ker qdp|_{\partial DT^*\Sigma}$ to itself, and is compatible with $\omega|_{\partial DT^*\Sigma}$.

Consider a smooth assignment 
$D_{m}\mapsto J_{D_{m}}$, where $D_{m}\in \mathcal{A}_{m}$, such that:
\begin{enumerate}
    \item[(J1)] on each fiber $\pi_{D_{m}}^{-1}(p)=\{p\}\times T^*\Sigma$, $J_{D_{m}}$ restricts to an element of $\mathcal{J}_{T^*\Sigma}$;
    \item[(J2)] $J_{D_{m}}$ projects holomorphically onto $D_{m}$;
    \item[(J3)] over each strip-like end $[0,\infty)_{s_i}\times[0,1]_{t_i}$, for $s_i$ sufficiently positive, or $(-\infty,0]_{s_0}\times[0,1]_{t_0}$, for $s_0$ sufficiently negative, $J_{D_{m}}$ is invariant in the $s_i$-direction and takes $\partial_{s_i}$ to $\partial_{t_i}$; when $m=1$, $J_{D_1}$ is invariant under $\mathbb{R}$-translation of the base and takes $\partial_{s_i}$ to $\partial_{t_i}$.
    \item[(J4)] consistent with the boundary strata.
\end{enumerate}

We abuse the notation and call such a smooth assignment $\{D_m\mapsto J_{D_m},\,m\geq1\}$ {\em a choice of perturbation data}.

\begin{lemma}\label{lemma:consistentperturbation}
    Suppose $J_{D_m}$ has been chosen for $m\leq m_0$, and satisfies the above conditions, then it can be extended to a choice of perturbation data for all $m\geq1$.
\end{lemma}
\begin{proof}
    One can inductively construct $\{J_{D_m},\,m\geq1\}$. We refer the reader to standard contexts such as \cite{seidel2008fukaya}.
\end{proof}

Let $\mathcal{R}(\boldsymbol{x}_{1},\ldots,\boldsymbol{x}_{m},\boldsymbol{x}_{0})$ denote the moduli space of maps
\begin{equation*}
    u\colon (\dot F,j)\to(D_{m}\times T^*\Sigma,J_{D_{m}}),
\end{equation*}
where $(F,j)$ is a compact Riemann surface with boundary, $\mathbf{w}_{0},\dots,\mathbf{w}_{m}$ are disjoint $\kappa$-tuples of boundary marked points of $F$, $\dot F=F\setminus\cup_i \mathbf{w}_i$, and $D_{m}\in \mathcal{A}_{m}$, so that $u$ satisfies
\begin{align}
    \left\{
        \begin{array}{ll}
            \text{$du\circ j=J_{D_{m}}\circ du$;}\\
            \text{$\pi_{T^*\Sigma}\circ u(z)\in \phi^{m-i}_{H_V}(T^*_{\mathbf{q}}\Sigma)$ if $\pi_{D_{m}}\circ u(z)\in\partial_{i}D_{m}$, $i=0,\dots,m$;}\\
            \text{$\pi_{T^*\Sigma}\circ u$ tends to $\boldsymbol{x}_i$ as $s_i\to+\infty$ for $i=1,\dots,m$;}\\
            \text{$\pi_{T^*\Sigma}\circ u$ tends to $\boldsymbol{x}_0$ as $s_0\to-\infty$;}\\
            \text{$\pi_{D_{m}}\circ u$ is a $\kappa$-fold branched cover of $D_{m}$.}
        \end{array}
    \right.
\end{align}
Transversality can be achieved using a standard Sard-Smale argument,
which ensures that the $0$-dimensional and $1$-dimensional moduli spaces are smooth manifolds and admit natural compactifications.
By counting rigid elements in the $0$-dimensional moduli spaces $\mathcal{R}^{\op{ind}=0}(\boldsymbol x_{1},\ldots,\boldsymbol x_{m},\boldsymbol x_{0})$,
we define 
\begin{align}
    \mu^m\colon &CW(T^*_\mathbf{q}\Sigma)\otimes\dots\otimes CW(T^*_\mathbf{q}\Sigma)\to CW(T^*_\mathbf{q}\Sigma), \label{eq-Ainfty} \\
    &\boldsymbol{x}_{1} \otimes \cdots \otimes \boldsymbol{x}_{m} \mapsto \sum_{\substack{\boldsymbol{x}_{0}\\u \in \mathcal{R}(\boldsymbol{x}_{1},\ldots,\boldsymbol{x}_{m},\boldsymbol{x}_{0})}} (-1)^{{\natural(u)}}\cdot c^{2\langle u,\star\rangle}\cdot \hbar^{\kappa-\chi(u)}\cdot[\boldsymbol{x}_{0}]. \label{eq-Ainftyformula}
\end{align}
for $m\geq1$.
By Lemma \ref{lemma:consistentperturbation} and a standard analysis of the boundary strata of $1$-dimensional moduli spaces, 
we may show that the maps $\mu^m$ satisfy the $A_{\infty}$ equations.
See \cite[Section 2.2]{honda2022jems} for more details.

\begin{remark}
    There is some subtlety when defining $A_\infty$-operations for wrapped HDHF. 
    We omit the details and refer the reader to \cite[Section 3]{abouzaid2010geometric} or \cite[Section 2.2]{honda2022jems} for the usual rescaling argument.
\end{remark}

\begin{remark}
    The convention for the $A_{\infty}$ multiplication \eqref{eq-Ainftyformula} is that the order goes from the left to the right,
    which is opposite to the convention in, e.g., \cites{abouzaid2010geometric, seidel2008fukaya, honda2022jems},
    but rather follows \cites{ganatra2020covariantly,ganatra2024sectorial}.
    We use the current convention, since it seems to better fit the topological description from the point of view of concatenation of paths.
\end{remark}


Next we proceed to construct a left $A_{\infty}$ $CW(T^*_{\mathbf{q}}\Sigma)$-module structure on the chain complex $CW(T^*_{\mathbf{q}}\Sigma, T^{*}_{\alpha}\Sigma)$.
Given two $\kappa$-tuples of chords $\boldsymbol{y}_{0},\boldsymbol{y}_1$ starting from $T^{*}_{\mathbf{q}}\Sigma$ ending on $T^{*}_{\alpha}\Sigma$,
and $\boldsymbol{x}_1,\dots,\boldsymbol{x}_{m-1}$ $\kappa$-tuples of chords starting from and ending on $T^{*}_{\mathbf{q}}\Sigma$, $m\geq1$, we define $\mathcal{M}(\boldsymbol{x}_{1},\dots,\boldsymbol{x}_{m-1},\boldsymbol{y}_1,\boldsymbol{y}_0)$ to be the moduli space of maps
\begin{equation*}
    u\colon (\dot F,j)\to(D_{m}\times T^*\Sigma,J_{D_{m}}),
\end{equation*}
where everything is the same as the setup of $\mathcal{R}(\boldsymbol{x}_{1},\ldots,\boldsymbol{x}_{m},\boldsymbol{x}_{0})$,
except that $u$ satisfies the following new set of conditions:
\begin{align}
    \label{floer-condition}
    \left\{
        \begin{array}{ll}
            \text{$du\circ j=J_{D_{m}}\circ du$;}\\
            \text{$\pi_{T^*\Sigma}\circ u(z)\in \phi^{m-i}_{H_V}(T^*_{\mathbf{q}}\Sigma)$ if $\pi_{D_{m}}\circ u(z)\in\partial_{i}D_{m}$, $i=0,\dots,m-1$;}\\
            \text{$\pi_{T^*\Sigma}\circ u(z)\in T^*_\alpha\Sigma$ if $\pi_{D_{m}}\circ u(z)\in\partial_{m}D_{m}$;}\\
            \text{$\pi_{T^*\Sigma}\circ u$ tends to $\boldsymbol{x}_i$ as $s_i\to+\infty$ for $i=1,\dots,m-1$;}\\
            \text{$\pi_{T^*\Sigma}\circ u$ tends to $\boldsymbol{y}_1$ as $s_m\to+\infty$;}\\
            \text{$\pi_{T^*\Sigma}\circ u$ tends to $\boldsymbol{y}_0$ as $s_0\to-\infty$;}\\
            \text{$\pi_{D_{m}}\circ u$ is a $\kappa$-fold branched cover of $D_{m}$.}
        \end{array}
    \right.
\end{align}
With the identification of each strip-like end $e_i$, $i=1,\dots,m$, with $[0,\infty)_{s_i}\times[0,1]_{t_i}$, the fourth condition means that $u$ maps the neighborhoods of the punctures of $\mathbf{w}_i$ asymptotically to the Reeb chords $[0,1]_{t_i}\times \boldsymbol{x}_i$ as $s_i\to +\infty$. The fifth and sixth conditions are similar. 
We also require that $\pi_{D_{m}}\circ u$ maps $\mathbf{w}_i$ to $w_i$.

\begin{lemma}
\label{lemma-ind}
    Fixing a generic choice of perturbation data, $\mathcal{M}(\boldsymbol{x}_{1},\dots,\boldsymbol{x}_{m-1},\boldsymbol{y}_1,\boldsymbol{y}_0)$ is a smooth manifold of dimension $m-2$.
\end{lemma}
\begin{proof}
    By the same proof of \cite[Theorem 2.3]{honda2022jems}, $\mathcal{M}(\boldsymbol{x}_{1},\dots,\boldsymbol{x}_{m-1},\boldsymbol{y}_1,\boldsymbol{y}_0)$ is a smooth manifold of dimension $|\boldsymbol{y}_0|-|\boldsymbol{x}_1|-\dots-|\boldsymbol{x}_{m-1}|-|\boldsymbol{y}_1|+m-2$ for a generic choice of $J_{D_m}$.
    By assumption of $\mathbf{q}$ and $\alpha$, all generators of $CW(T_{\mathbf{q}}^*\Sigma)$ and $CW(T_{\mathbf{q}}^*\Sigma,T_{\alpha}^*\Sigma)$ are of degree 0, hence $|\boldsymbol{y}_0|=|\boldsymbol{y}_1|=|\boldsymbol{x}_1|=\dots=|\boldsymbol{x}_{m-1}|=0$.
\end{proof}

For the well-definedness of the coefficient ring, we have to consider braid skein algebras on $T^*_\alpha\Sigma$, and a fixed $\kappa$-tuple of basepoints on $T^*_\alpha\Sigma$ is necessary. 
Given $\mathbf{p}_0,\mathbf{p}_1\in\mathrm{UConf}_{\kappa}(T^*_\alpha\Sigma)$, consider the space $\mathrm{Br}_{\kappa}(T^*_\alpha\Sigma,\mathbf{p}_0,\mathbf{p}_1)$ which consists of braids that start from $\mathbf{p}_0$ and end on $\mathbf{p}_1$.
We define $\mathrm{BSk}^\circ_\kappa(T^*_\alpha\Sigma,\mathbf{p}_0,\mathbf{p}_1)$ as the quotient of the free module $\mathbb{Z}[s^{\pm1}][\mathrm{Br}_{\kappa}(T^*_\alpha\Sigma,\mathbf{p}_0,\mathbf{p}_0)]$ by the local relation (\ref{eq-skein'}).


\begin{definition}
    We define $CW(T^*\Sigma,\mathbf{q},\alpha,\mathbf{p}_0)$ to be the free $\mathbb{Z}\llbracket s^{\pm1}\rrbracket[c^{\pm1}]$-module generated by $\{([\gamma],\boldsymbol{y})\}$, for all $[\gamma]\in\op{BSk}^\circ_\kappa(T^*_\alpha\Sigma,\mathbf{p}_0,\boldsymbol{y}(1))$ and $\boldsymbol{y}\in CW(T_{\mathbf{q}}^*\Sigma,T_{\alpha}^*\Sigma)$, such that $\mathbb{Z}\llbracket s^{\pm1}\rrbracket$ naturally acts on the first factor.
    For notation convenience, we write $[\gamma,\boldsymbol{y}]\coloneqq ([\gamma],\boldsymbol{y})$.
\end{definition}

Fix a parametrization of the arc $\partial_m D_m$ from $w_m$ to $w_0$ by $\tau\colon [0,1] \to \partial_m D_m$.
Generically for all $u\in\mathcal{M}(\boldsymbol{x}_{1},\dots,\boldsymbol{x}_{m-1},\boldsymbol{y}_1,\boldsymbol{y}_0)$, $(\pi_{T^*\Sigma}\circ u)\circ(\pi_{D_m}\circ u)^{-1}\circ\tau(t)$ consists of $\kappa$ distinct points on $T^*_\alpha\Sigma$ for each $t\in[0,1]$ and hence gives a path in $\mathrm{UConf}_{\kappa}(T^*_\alpha\Sigma)$:
\begin{equation}
\label{eq-evaluation}
    \mathcal{E}_\alpha(u)\colon [0,1] \to \mathrm{UConf}_{\kappa}(T^*_\alpha\Sigma), \quad t\mapsto (\pi_{T^*\Sigma}\circ u)\circ(\pi_{D_m}\circ u)^{-1}\circ\tau(t),
\end{equation}
where $\mathcal{E}_\alpha(u)(0)=\boldsymbol{y}_1(1)$, $\mathcal{E}_\alpha(u)(1)=\boldsymbol{y}_0(1)$.


Consider the map
\begin{gather}
    \mu_{\alpha}^m\colon CW(T^*_\mathbf{q}\Sigma)^{\otimes m-1}\otimes CW(T^*\Sigma,\mathbf{q},\alpha,\mathbf{p}_0)\to CW(T^*\Sigma,\mathbf{q},\alpha,\mathbf{p}_0), \label{modulemap}\\ 
    \boldsymbol{x}_{1}\otimes\dots\otimes \boldsymbol{x}_{m-1}\otimes[\gamma,\boldsymbol{y}_{1}] \mapsto \sum_{\substack{\boldsymbol{y},u\in\\\mathcal{M}^{\op{ind}=0}(\boldsymbol{x}_1,\dots,\boldsymbol{x}_{m-1},\boldsymbol{y}_1,\boldsymbol{y})}} (-1)^{{\natural(u)}}\cdot c^{2\langle u,\star\rangle}\cdot \hbar^{\kappa-\chi(u)}\cdot[\gamma\cdot\mathcal{E}_\alpha(u),\boldsymbol{y}]. \label{modulemapformula}
\end{gather}
Here $\langle u,\star\rangle$ is the signed intersection number between $\pi_{T^*\Sigma}\circ u$ and $T^*_\star\Sigma$ in $T^*\Sigma$. See \cite[Section 5]{honda2022jems} for the details about the $c$-parameter, which corresponds to the marked point relation (\ref{eq-c}).
In the rest of this paper, we omit the discussion of the $c$-parameter since there is no extra behavior than that in \cite{honda2022jems}.

\begin{lemma}
     The maps $\mu_{\alpha}^{m}$ \eqref{modulemap} defined by \eqref{modulemapformula} endow $CW(T^*\Sigma,\mathbf{q},\alpha,\mathbf{p}_0)$ the structure of a left $A_{\infty}$-module over $CW(T^*_\mathbf{q}\Sigma)$.
     Moreover, the differential and higher order module structure maps vanish,
     so that $CW(T^*\Sigma,\mathbf{q},\alpha,\mathbf{p}_0)$ is a usual left module over the graded algebra $CW(T^*_\mathbf{q}\Sigma)$,
     which is concentrated in degree zero.
\end{lemma}
\begin{proof}
    By Lemma \ref{lemma-ind}, $\mathcal{M}(\boldsymbol{x}_{1},\dots,\boldsymbol{x}_{m-1},\boldsymbol{y}_1,\boldsymbol{y}_0)$ is of dimension 0 if and only if $m=2$, and is of dimension 1 if and only if $m=3$.
    By standard Gromov compactness results of wrapped Floer homology (see \cite[Lemma 2.6]{honda2022jems}), $\mathcal{M}(\boldsymbol{x}_{1},\dots,\boldsymbol{x}_{m-1},\boldsymbol{y}_1,\boldsymbol{y}_0)$ admits a compactification if $m\leq 3$. 
    The 0-dimensional boundary $\partial\overline{\mathcal{M}}(\boldsymbol{x}_1,\boldsymbol{x}_{2},\boldsymbol{y}_{1},\boldsymbol{y}_{0})$ consists of three pieces:
    \begin{enumerate}
        \item $\bigsqcup_{\boldsymbol{y}'}\mathcal{M}(\boldsymbol{x}_2,\boldsymbol{y}_{1},\boldsymbol{y}')\times\mathcal{M}(\boldsymbol{x}_1,\boldsymbol{y}',\boldsymbol{y}_0)$. This corresponds to the coefficient of $\boldsymbol{y}_0$ in $\mu_{\alpha}^2(\boldsymbol{x}_1,\mu_{\alpha}^2(\boldsymbol{x}_2,[\gamma,\boldsymbol{y}_1]))$;
        \item $\bigsqcup_{\boldsymbol{x}'}\mathcal{R}(\boldsymbol{x}_1,\boldsymbol{x}_{2},\boldsymbol{x}')\times\mathcal{M}(\boldsymbol{x}',\boldsymbol{y}_1,\boldsymbol{y}_0)$. Here $\mathcal{R}(\boldsymbol{x}_1,\boldsymbol{x}_{2},\boldsymbol{x}')$ is the moduli space involved in the definition of the $\mu^2$ map in (\ref{eq-Ainfty}). This corresponds to the coefficient of $\boldsymbol{y}_0$ in $\mu_{\alpha}^2(\mu^2(\boldsymbol{x}_1,\boldsymbol{x}_2),[\gamma,\boldsymbol{y}_1])$;
        \item the set $\partial_\alpha\overline{\mathcal{M}}(\boldsymbol{x}_1,\boldsymbol{x}_{2},\boldsymbol{y}_{1},\boldsymbol{y}_{0})$ of curves with a nodal degeneration along $T^*_\alpha\Sigma$. This corresponds to the symplectic skein relation discussed in \cite{ekholm2021skeins}. The outcome is that such degeneration causes no trouble in curve counting. We will discuss similar situations in details in the proof of Lemma \ref{lemma-homo}.
    \end{enumerate}

\end{proof}

We denote the homology of $CW(T^*\Sigma,\mathbf{q},\alpha,\mathbf{p}_0)$ (resp. $CW(T^*_\mathbf{q}\Sigma)$) by $HW(T^*\Sigma,\mathbf{q},\alpha,\mathbf{p}_0)$ (resp. $HW(T^*_\mathbf{q}\Sigma)$).

In this paper, we omit the details about the orientation of $\mathcal{M}(\boldsymbol{x}_{1},\dots,\boldsymbol{x}_{m-1},\boldsymbol{y}_1,\boldsymbol{y}_0)$ and related moduli spaces, and refer the reader to \cite[Section 4]{colin2020applications}.

\subsection{Symplectic skein relations at corners}
\label{section-local}
In this subsection, we give a symplectic explanation of the skein relation at corners in Equation (\ref{eq-corner}). 
Suppose $L_0,L_1$ are two Lagrangians in a symplectic 6-manifold $(X^6,\omega)$ such that $L_0,L_1$ have clean intersection at a 1-dimensional submanifold $K=L_0\cap L_1$. 

\begin{figure}[ht]
    \centering
    \includegraphics[width=15cm]{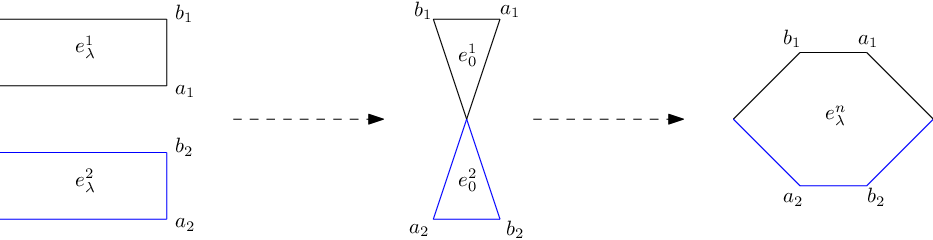}
    \caption{After resolving at $s_i=-\infty$, the strip like ends $e^1_\lambda\sqcup e^2_\lambda$ is replaced by $e^n_\lambda$.}
    \label{fig-strip-nodal}
\end{figure}

Let $S_\lambda$ be a 1-parameter family of Riemann surfaces with at least two boundary punctures $w_1,w_2$ for $\lambda\in[-\epsilon,\epsilon]$. Let $e^i_\lambda\simeq (-\infty,0]_{s_i}\times[0,1]_{t_i}$ be a smooth family of strip-like ends near $w_i$, $i=1,2$. 
See the left of Figure \ref{fig-strip-nodal}.
Suppose that $u_\lambda\colon S_\lambda\to X$ is a 1-parameter family of maps such that
\begin{align}
    \label{floer-condition-skein}
    \left\{
        \begin{array}{ll}
            \text{$du_\lambda\circ j=J\circ du_\lambda$;}\\
            \text{$u_\lambda(\{t_i=0\})\subset L_0$, $i=1,2$;}\\
            \text{$u_\lambda(\{t_i=1\})\subset L_1$, $i=1,2$;}\\
            \text{$u_\lambda$ tends to $K$ as $s_i\to-\infty$, $i=1,2$.}
        \end{array}
    \right.
\end{align}
Suppose $u_\lambda(\{s_1=-\infty\})=u_\lambda(\{s_2=-\infty\})$ at $\lambda=0$. 
Then there are two possibilities, where there is an additional 1-parameter family of curves in the upper case of Figure \ref{fig-corner-skein}.

\begin{figure}[ht]
    \centering
    \includegraphics[width=12cm]{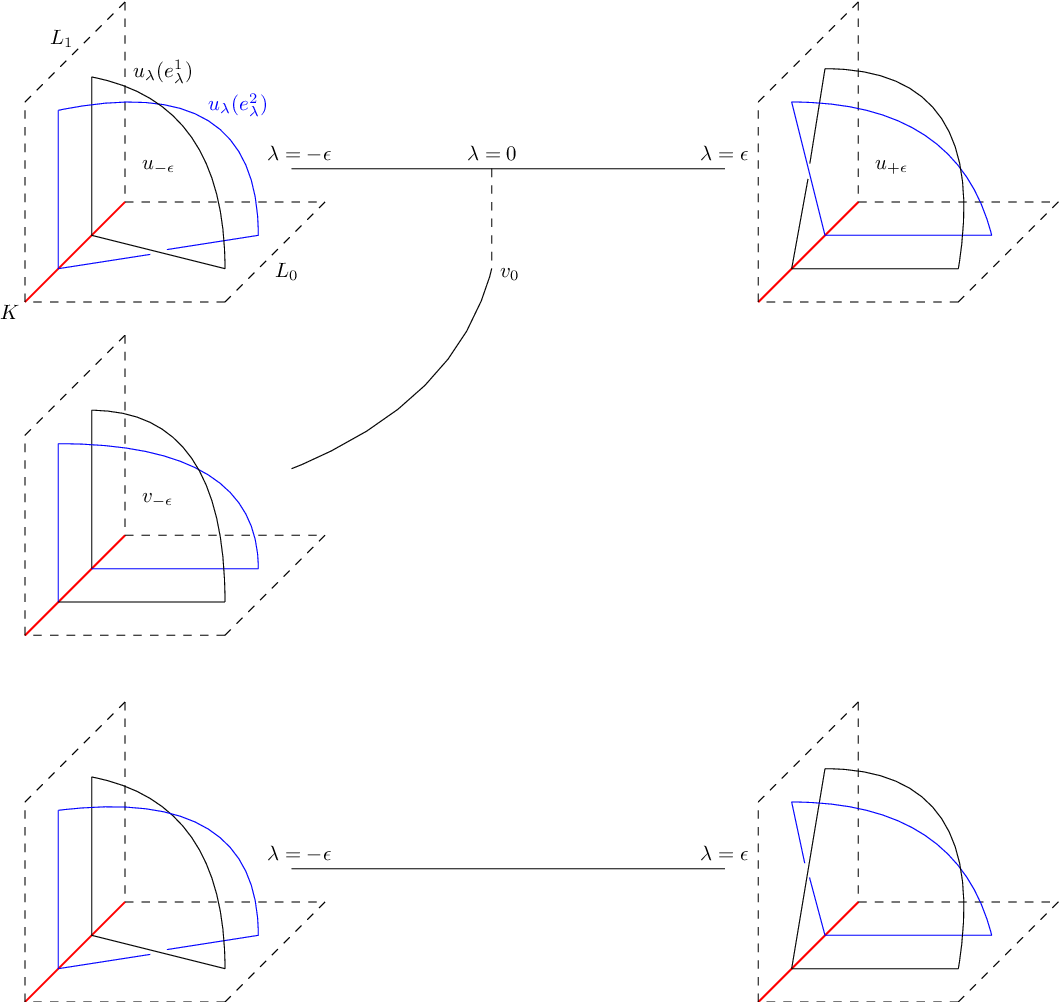}
    \caption{$u_\lambda$ restricted to $e^1_\lambda\sqcup e^2_\lambda$. In the upper case, after resolving $u_0$ we get a new family of nodal curves, where $e^1_\lambda\sqcup e^2_\lambda$ is replaced by $e^n_\lambda$.}
    \label{fig-corner-skein}
\end{figure}

We use a local model to explain Figure \ref{fig-corner-skein}. For a more rigorous proof, see the standard gluing techniques as in \cite{ekholm2021skeins}.

Consider $\mathbb{C}^3_{z_1,z_2,z_3}$, $z_i=x_i+jy_i$, $i=1,2,3$, with the standard complex structure $j$.
Let $\pi_{z_i}\colon \mathbb{C}^3\to\mathbb{C}_{z_i}$ be the projection map.

For the upper case of Figure \ref{fig-corner-skein}, we define the family of curves $u_\lambda\colon e^1_\lambda\sqcup e^2_\lambda\to\mathbb{C}^3$ in the split form as in Figure \ref{fig-corner-skein-model}, where $\pi_{z_1}\circ u_\lambda$ is of degree 2 over the triangle, $\pi_{z_2}\circ u_\lambda$ is of degree 2 near the origin, and $\pi_{z_3}\circ u_\lambda$ is a constant map restricted to $e^1_\lambda$ and $e^2_\lambda$.
As $\lambda$ varies, $\pi_{z_1}\circ u_\lambda$ and $\pi_{z_2}\circ u_\lambda$ remain unchanged. The constant maps $\pi_{z_3}\circ u_\lambda(e^1_\lambda)$ and $\pi_{z_3}\circ u_\lambda(e^2_\lambda)$ are moving in opposite directions as $\lambda$ increases and meet each other at $\lambda=0$. 
We resolve the domain $e^1_0\sqcup e^2_0$ to get $e^n_0$ as the middle of Figure \ref{fig-strip-nodal}, and correspondingly get a nodal curve $u^n_0\colon e^n_0\to\mathbb{C}^3$. 
There is then a family of curves $u^n_\lambda\colon e^n_\lambda\to\mathbb{C}^3$, $\lambda\geq0$, parametrized by the position of branched points of $\pi_{z_i}\circ u^n_\lambda$, $i=1,2$, which lies over the red arcs in Figure \ref{fig-corner-skein-model}. 
When $\lambda=0$, the branched points are at the origin in both $z_1,z_2$-directions.

\begin{figure}[ht] 
    \centering
    \includegraphics[width=11cm]{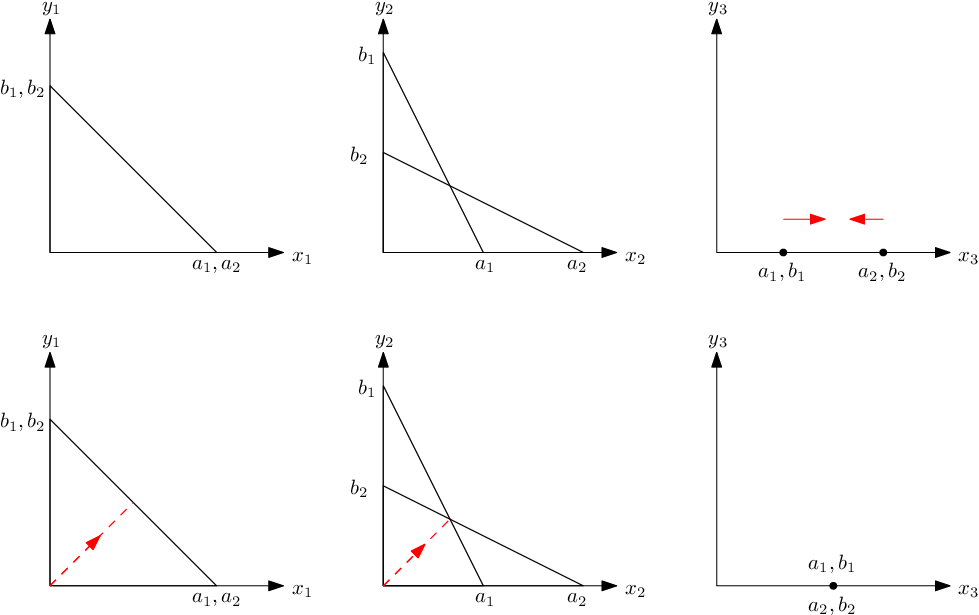}
    \caption{As the image of $e^1_\lambda$ and $e^2_\lambda$ are getting closer on the $z_3$-direction, a new family of curves $u^n_\lambda$, $\lambda\geq0$ occurs, where $u^n_0$ is the nodal curve resolved from $u_0$. The projection of $u^n_\lambda$ to $z_i$-direction, $i=1,2$, has a double branched point on the red dashed arc, which moves away from the origin as $\lambda$ increases.}
    \label{fig-corner-skein-model}
\end{figure}

For the lower case of Figure \ref{fig-corner-skein}, we define the family of curves $u_\lambda\colon e^1_\lambda\sqcup e^2_\lambda\to\mathbb{C}^3$ in the split form as in Figure \ref{fig-corner-skein-model-other}. 
As $\pi_{z_3}\circ u_\lambda(e_\lambda^1)$ and $\pi_{z_3}\circ u_\lambda(e_\lambda^2)$ meet each other at $\lambda=0$, we similarly resolve the domain $e^1_0\sqcup e^2_0$ to $e^n_0$ and get the nodal curve $u^n_0\colon e^n_0\to\mathbb{C}^3$.
However, the new family $u^n_\lambda$ cannot exist for $\lambda>0$: Observe that the projection $\pi_{z_1}\circ u^n_\lambda$ naturally defines an involution map on the hexagon $e^n_\lambda$ by the deck transformation of $\pi_{z_1}$. 
If $\pi_{z_2}\circ u^n_\lambda$ has a branched point away from the origin, e.g., on the red dashed arc in Figure \ref{fig-corner-skein-model-other}, then the boundary component from $a_1$ to $b_1$ is longer than that from $a_2$ to $b_2$ on the domain $e^n_\lambda$, which contradicts with the involution condition.

\begin{figure}[ht]
    \centering
    \includegraphics[width=11cm]{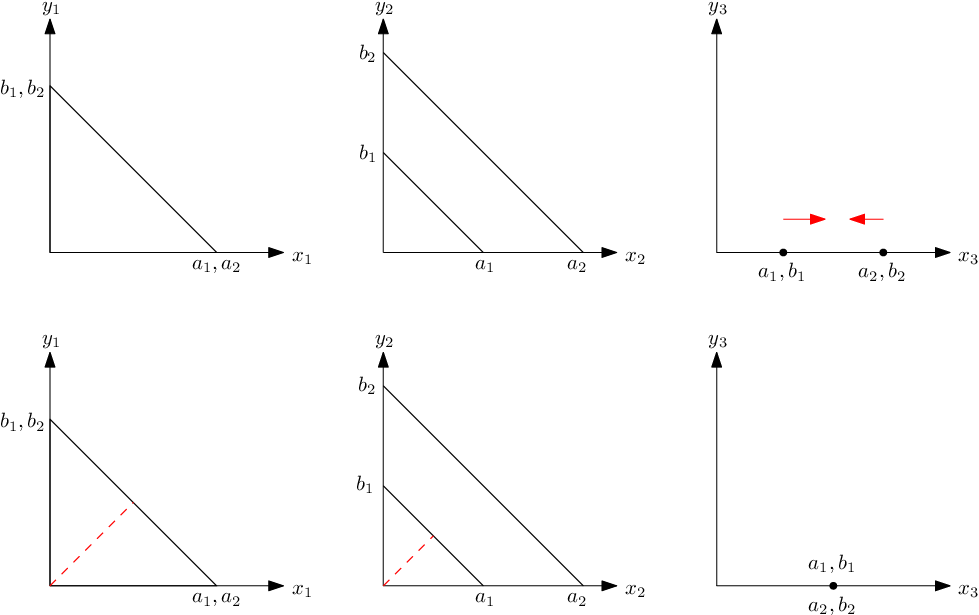}
    \caption{As the image of $e^1_\lambda$ and $e^2_\lambda$ are getting closer on the $z_3$-direction, there cannot be a new family of resolved curves.}
    \label{fig-corner-skein-model-other}
\end{figure}

\subsection{The isomorphism}



Recall that the main result of \cite{honda2022jems} is the following:

\begin{theorem}[{\cite[Theorem 1.4. Corollary 1.5]{honda2022jems}}]
    For $\Sigma \neq S^2$, the $A_{\infty}$-algebra $CW(T^{*}_{\mathbf{q}}\Sigma)$ is an ordinary algebra concentrated in degree $0$.
    There is an isomorphism
    \begin{equation}\label{dahaiso}
        \mathcal{F}: HW(T^{*}_{\mathbf{q}}\Sigma) \to \mathrm{BSk}_\kappa(\Sigma,\mathbf{q}) \otimes_{\Z[\hbar]}\Z\llbracket\hbar\rrbracket,
    \end{equation}
    where the latter is furthermore isomorphic to $\ddot{H}_{\kappa}|_{\hbar=s-s^{-1}}\otimes_{\Z[\hbar]}\Z\llbracket\hbar\rrbracket$.
\end{theorem}

See \cite[\S6]{honda2022jems} for details in the construction of the map $\mathcal{F}$ \eqref{dahaiso}.
In this subsection, we generalize the construction of the map \eqref{dahaiso} to a map from the HDHF module $HW(T^{*}\Sigma,\mathbf{q},\alpha,\mathbf{p}_{0})$ to the braid Skein module $\mathrm{BSk}_\kappa(\Sigma,\mathbf{q},\alpha,\mathbf{p}_{0})$.

We begin by introducing several moduli spaces, which can be seen as generalizations of those considered in \cite[\S6]{honda2022jems}.
Given $\boldsymbol{y}\in CW(T_{\mathbf{q}}^*\Sigma,T_{\alpha}^*\Sigma)$, we define $\mathcal{H}(\boldsymbol{y},\alpha,\mathbf{q})$ to be the moduli space of maps
\begin{equation*}
    u\colon (\dot F,j)\to(D_{2}\times T^*\Sigma,J_{D_2}),
\end{equation*}
where $(F,j)$ is a compact Riemann surface with boundary, $\mathbf{w}_{0},\mathbf{w}_{1},\mathbf{w}_{2}$ are disjoint $\kappa$-tuples of boundary marked points of $F$, $\dot F=F\setminus\cup_i \mathbf{w}_i$, and $D_{2}\in \mathcal{A}_{2}$, so that $u$ satisfies
\begin{align}
    \label{floer-condition-T1}
    \left\{
        \begin{array}{ll}
            \text{$du\circ j=J_{D_2}\circ du$;}\\
            \text{$\pi_{T^*\Sigma}\circ u(z)\in \phi^1_{H_V}(T^*_{\mathbf{q}}\Sigma)$ if $\pi_{D_{2}}\circ u(z)\in\partial_{1}D_{2}$;}\\
            \text{$\pi_{T^*\Sigma}\circ u(z)\in T^*_\alpha\Sigma$ if $\pi_{D_{2}}\circ u(z)\in\partial_{2}D_{2}$;}\\
            \text{$\pi_{T^*\Sigma}\circ u(z)\in\Sigma$ if $\pi_{D_{2}}\circ u(z)\in\partial_{0}D_{2}$;}\\
            \text{$\pi_{T^*\Sigma}\circ u$ tends to $\mathbf{q}$ as $s_1\to+\infty$;}\\
            \text{$\pi_{T^*\Sigma}\circ u$ tends to $\boldsymbol{y}$ as $s_2\to+\infty$}\\
            \text{$\pi_{T^*\Sigma}\circ u$ tends to $\alpha$ as $s_0\to-\infty$;}\\
            \text{$\pi_{D_{2}}\circ u$ is a $\kappa$-fold branched cover of $D_{2}$.}
        \end{array}
    \right.
\end{align}

Given another $\kappa$-tuple of basepoints $\mathbf{q}'\subset\Sigma$, we define $\mathcal{H}(\boldsymbol{y},\mathbf{q}',\mathbf{q})$ similarly, replacing $T^*_\alpha\Sigma$ and $\alpha$ by $T^*_{\mathbf{q}'}\Sigma$ and $\mathbf{q}'$ respectively in (\ref{floer-condition-T1}).

\begin{lemma}
    Fixing a generic choice of perturbation data, $\mathcal{H}(\boldsymbol{y},\alpha,\mathbf{q})$ and $\mathcal{H}(\boldsymbol{y},\mathbf{q}',\mathbf{q})$ are smooth manifolds of dimension 0 and admit a compactification.
\end{lemma}
\begin{proof}
    This follows from Lemma 6.1 and Lemma 6.2 of \cite{honda2022jems}.
\end{proof}

\begin{figure}[ht]
    \centering
    \includegraphics[width=12cm]{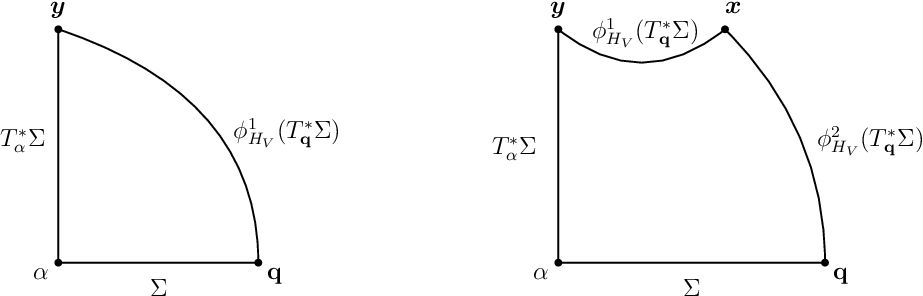}
    \caption{The $A_\infty$ base direction for $\mathcal{H}(\boldsymbol{y},\alpha,\mathbf{q})$ (left) and $\mathcal{H}(\boldsymbol{y},\alpha,\mathbf{q},\boldsymbol{x})$ (right).}
    \label{fig-T1T2}
\end{figure}

Given $\boldsymbol{y}\in CW(T_{\mathbf{q}}^*\Sigma,T_{\alpha}^*\Sigma)$ and $\boldsymbol{x}\in CW(T_{\mathbf{q}}^*\Sigma)$, we define $\mathcal{H}(\boldsymbol{y},\alpha,\mathbf{q},\boldsymbol{x})$ to consist of maps
\begin{equation*}
    u\colon (\dot F,j)\to(D_{3}\times T^*\Sigma,J_{D_3}),
\end{equation*}
where $(F,j)$ is a compact Riemann surface with boundary, $\mathbf{w}_{0},\dots,\mathbf{w}_{3}$ are disjoint $\kappa$-tuples of boundary marked points of $F$, $\dot F=F\setminus\cup_i \mathbf{w}_i$, and $D_{3}\in \mathcal{A}_{3}$, so that $u$ satisfies
\begin{align}
    \label{floer-condition-T2}
    \left\{
        \begin{array}{ll}
            \text{$du\circ j=J_{D_3}\circ du$;}\\
            \text{$\pi_{T^*\Sigma}\circ u(z)\in \phi^2_{H_V}(T^*_{\mathbf{q}}\Sigma)$ if $\pi_{D_{3}}\circ u(z)\in\partial_{1}D_{3}$;}\\
            \text{$\pi_{T^*\Sigma}\circ u(z)\in \phi^1_{H_V}(T^*_{\mathbf{q}}\Sigma)$ if $\pi_{D_{3}}\circ u(z)\in\partial_{2}D_{3}$;}\\
            \text{$\pi_{T^*\Sigma}\circ u(z)\in T^*_\alpha\Sigma$ if $\pi_{D_{3}}\circ u(z)\in\partial_{3}D_{3}$;}\\
            \text{$\pi_{T^*\Sigma}\circ u(z)\in\Sigma$ if $\pi_{D_{3}}\circ u(z)\in\partial_{0}D_{3}$;}\\
            \text{$\pi_{T^*\Sigma}\circ u$ tends to $\mathbf{q}$ as $s_1\to+\infty$;}\\
            \text{$\pi_{T^*\Sigma}\circ u$ tends to $\boldsymbol{x}$ as $s_2\to+\infty$}\\
            \text{$\pi_{T^*\Sigma}\circ u$ tends to $\boldsymbol{y}$ as $s_3\to+\infty$}\\
            \text{$\pi_{T^*\Sigma}\circ u$ tends to $\alpha$ as $s_0\to-\infty$;}\\
            \text{$\pi_{D_{3}}\circ u$ is a $\kappa$-fold branched cover of $D_{3}$.}
        \end{array}
    \right.
\end{align}

\begin{figure}[ht]
    \centering
    \includegraphics[width=14cm]{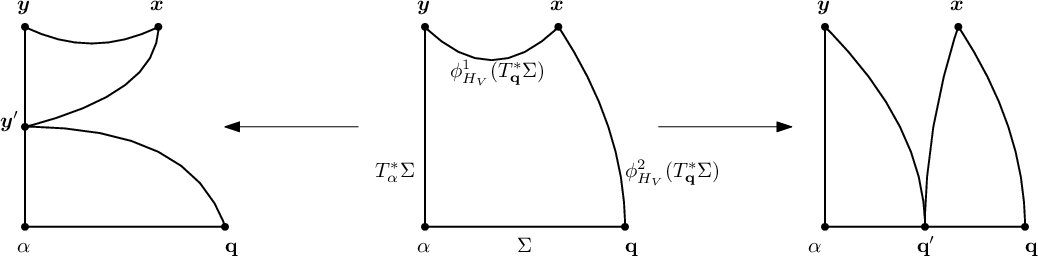}
    \caption{Codimension-1 degeneration of $\mathcal{H}(\boldsymbol{y},\alpha,\mathbf{q},\boldsymbol{x})$ in the $A_\infty$ base direction: concatenation of curves in $\mathcal{M}(\boldsymbol{y}',\boldsymbol{x},\boldsymbol{y})$ and $\mathcal{H}(\boldsymbol{y}',\alpha,\mathbf{q})$ (left); concatenation of curves in $\mathcal{H}(\boldsymbol{y},\alpha,\mathbf{q}')$ and $\mathcal{H}(\boldsymbol{x},\mathbf{q}',\mathbf{q})$ (right).}
    \label{fig-T2-deg}
\end{figure}

\begin{figure}[ht]
    \centering
    \includegraphics[width=13cm]{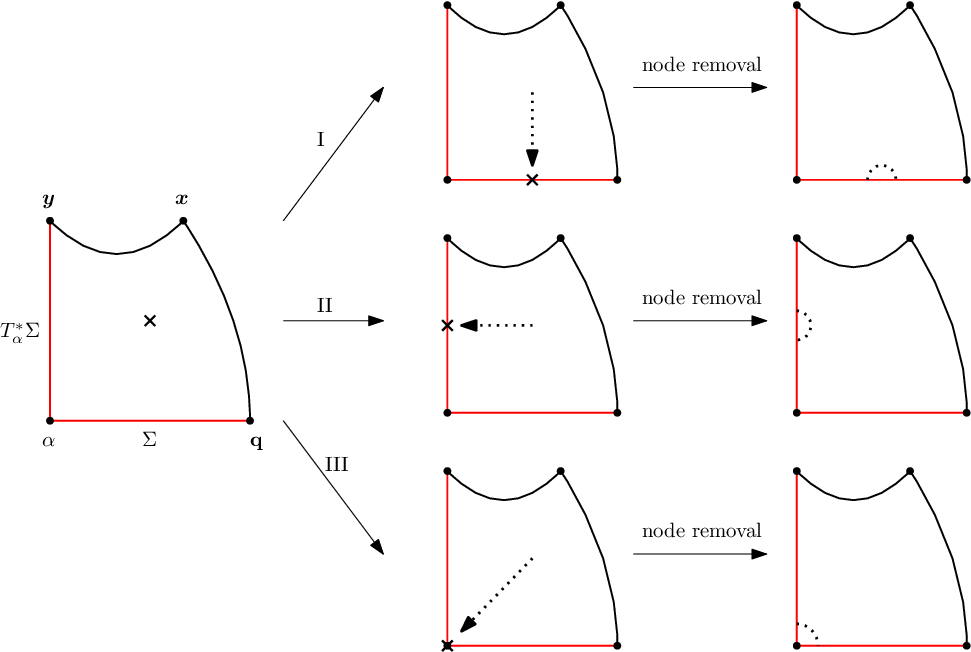}
    \caption{Three types of codimension-1 nodal degeneration: (I) a branched point approaches $\Sigma$; (II) a branched point approaches $T^*_\alpha\Sigma$; (III) a branched point approaches $\alpha$.}
    \label{fig-T2-nodal}
\end{figure}

\begin{lemma}
    Fixing a generic choice of perturbation data, $\mathcal{H}(\boldsymbol{y},\alpha,\mathbf{q},\boldsymbol{x})$ is a smooth manifold of dimension 1 and admits a compactification.
\end{lemma}
\begin{proof}
    Similar to Lemma 6.4 of \cite{honda2022jems}, since $|\boldsymbol{y}|=|\boldsymbol{x}|=0$, for a generic choice of perturbation data, $\mathcal{H}(\boldsymbol{y},\alpha,\mathbf{q},\boldsymbol{x})$ is a smooth manifold of dimension 1. By standard Gromov compactness theorem, it admits a compactification $\overline{\mathcal{H}}(\boldsymbol{y},\alpha,\mathbf{q},\boldsymbol{x})$. We then need to analyze its boundary behavior.

    Since the ambient symplectic manifold and all Lagrangians we consider are exact, there is no bubble with positive energy. We still need to exclude ghost bubbles, i.e., bubbles with zero energy. On the $A_\infty$ direction $D_3$, there could be two or more branched points collapsing inside $D_3$ (ghost bubble spheres) or on $\partial_0 D_3\sqcup\partial_3 D_3$ (ghost bubble disks). These are excluded by \cite{ekholm2022ghostbubblecensorship}. 
    The remaining case is that two or more branched points may approach $w_0$ simultaneously, which forms ghost bubble disks attached to $w_0$.

    Recall the local model in Section \ref{section-local}, where we consider $\mathbb{C}^3_{z_1,z_2,z_3}$, $z_i=x_i+jy_i$, $i=1,2,3$, with the standard complex structure $j$ and let $\pi_{z_i}\colon \mathbb{C}^3\to\mathbb{C}_{z_i}$ be the projection map.
    We assume that locally $\partial_0 D_3$ (resp. $\partial_3 D_3$) corresponds to $\mathbb{C}_{z_1}\cap\{y_1=0\}$ (resp. $\mathbb{C}_{z_1}\cap\{x_1=0\}$); $T^*\Sigma$ corresponds to $\mathbb{C}^2_{z_2,z_3}$, $\Sigma$ corresponds to $\mathbb{C}^2_{z_2,z_3}\cap\{y_2=y_3=0\}$, and $T^*_\alpha\Sigma$ corresponds to $\mathbb{C}^2_{z_2,z_3}\cap\{x_2=y_3=0\}$.
    Consider a 1-parameter family of holomorphic curves $u_\lambda\in{\mathcal{H}}(\boldsymbol{y},\alpha,\mathbf{q},\boldsymbol{x})$ in the local model.
    Then $(\pi_{z_1}\circ u_\lambda)^2$ (resp. $(\pi_{z_2}\circ u_\lambda)^2$) is a holomorphic map to $\mathbb{C}_{z_1}$ (resp. $\mathbb{C}_{z_2}$) with boundary condition on $\mathbb{C}_{z_1}\cap\{y_1=0\}$ (resp. $\mathbb{C}_{z_2}\cap\{y_2=0\}$). 
    Now that two or more branched points approaching $w_0$ simultaneously creates a boundary ghost bubble attached to $0\in\mathbb{C}^2_{z_1,z_2}$. 
    Denote the nonconstant component of the limiting curve by $v:(F,\partial F)\to (\mathbb{C}^3,\mathbb{R}^3)$.
    By \cite[Theorem 1.1]{ekholm2022ghostbubblecensorship}, creating a boundary ghost bubble along $\mathbb{R}^2\subset\mathbb{C}^2$ imposes either
    \begin{enumerate}
        \item a coincident and tangential boundary condition on $v$ at some $p_1,p_2\in\partial F$ which is of codimension 1. The projection $\pi_{z_3}\circ v(p_1)=\pi_{z_3}\circ v(p_2)$ imposes an additional codimension 1, hence $v$ is of negative index;
        \item a coincident boundary condition on $v$ at 3 points $p_1,p_2,p_3\in\partial F$. The projection $\pi_{z_3}\circ v(p_1)=\pi_{z_3}\circ v(p_2)=\pi_{z_3}\circ v(p_3)$ imposes codimension 2, hence $v$ is of negative index.
    \end{enumerate}
    In both cases, a transverse $v$ does not exist. Therefore, in a generic 1-parameter family of curves, two or more branched points cannot approach $w_0$ simultaneously.
    For the general case of almost complex structure $J$, we refer the reader to \cite{ekholm2022ghostbubblecensorship} for details.
    
    Now a 1-parameter family of curves in $\mathcal{H}(\boldsymbol{y},\alpha,\mathbf{q},\boldsymbol{x})$ may either
    \begin{enumerate}
        \item limit to broken curves by pinching boundaries of $D_3$ as shown in Figure \ref{fig-T2-deg};
        \item limit to nodal curves by letting one branched point approach $\partial_0 D_3$, $\partial_3 D_3$ or $w_0$ as shown in Figure \ref{fig-T2-nodal}.
    \end{enumerate}
\end{proof}

\begin{figure}[ht]
    \centering
    \includegraphics[width=9cm]{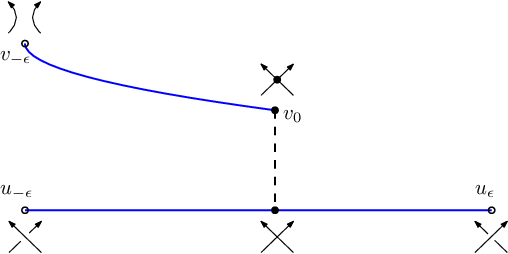}
    \caption{}
    \label{fig-skein-brane}
\end{figure}

We denote the subset of $\partial\overline{\mathcal{H}}(\boldsymbol{y},\alpha,\mathbf{q},\boldsymbol{x})$ corresponding to I, II, and III of Figure \ref{fig-T2-nodal} by $\partial_\Sigma\overline{\mathcal{H}}(\boldsymbol{y},\alpha,\mathbf{q},\boldsymbol{x})$, $\partial_\alpha\overline{\mathcal{H}}(\boldsymbol{y},\alpha,\mathbf{q},\boldsymbol{x})$, and $\partial_c\overline{\mathcal{H}}(\boldsymbol{y},\alpha,\mathbf{q},\boldsymbol{x})$, respectively.

Similar to (\ref{eq-evaluation}), for each $u\in\mathcal{H}(\boldsymbol{y},\alpha,\mathbf{q})$ or generic $u\in\mathcal{H}(\boldsymbol{y},\alpha,\mathbf{q},\boldsymbol{x})$, its boundary evaluation on $T^*_\alpha\Sigma$ defines a map $\mathcal{E}_\alpha(u)\in C^0([0,1],\op{UConf}_\kappa(T^*_\alpha\Sigma))$; similarly, the boundary evaluation on $\Sigma$ defines a map $\mathcal{E}_\Sigma(u)\in C^0([0,1],\op{UConf}_\kappa(\Sigma))$.

We define
\begin{gather} 
    \mathcal{F}_{\alpha}\colon HW(T^*\Sigma,\mathbf{q},\alpha,\mathbf{p}_0)\to \op{BSk}_\kappa(\Sigma,\mathbf{q},\alpha,\mathbf{p}_0)\otimes_{\mathbb{Z}[\hbar]}\mathbb{Z}\llbracket\hbar\rrbracket,\label{eq-F}\\
    [\gamma,\boldsymbol{x}]\mapsto \sum_{u\in\mathcal{H}(\boldsymbol{x},\alpha,\mathbf{q})} (-1)^{\natural(u)}\cdot c^{2\langle u,\star\rangle}\cdot\hbar^{\kappa-\chi(u)}\cdot [(\gamma\cdot\mathcal{E}_\alpha(u),\mathcal{E}_\Sigma(u))].\nonumber
\end{gather}

In a manner similar to the definition of $\mathcal{F}_{\alpha}$, we define the evaluation map
\begin{gather}
    \mathcal{G}\colon\mathcal{H}(\boldsymbol{y},\alpha,\mathbf{q},\boldsymbol{x})\to\op{BSk}_\kappa(\Sigma,\boldsymbol{y},\alpha,\mathbf{q})\otimes_{\mathbb{Z}[\hbar]}\mathbb{Z}\llbracket\hbar\rrbracket,\label{eq-G}\\
    u\mapsto (-1)^{\natural(u)}\cdot c^{2\langle u,\star\rangle}\cdot\hbar^{\kappa-\chi(u)}\cdot [(\mathcal{E}_\alpha(u),\mathcal{E}_\Sigma(u))].\nonumber
\end{gather}

\begin{lemma}
\label{lemma-homo}
  The map $\mathcal{F}_{\alpha}$ is a homomorphism of modules.
\end{lemma}
\begin{proof}
    It suffices to show that
    \begin{equation*}
        \mathcal{F}_{\alpha}(\mu_{\alpha}^2(\boldsymbol{x},[\gamma,\boldsymbol{y}]))=\mathcal{F}(\boldsymbol{x})\mathcal{F}_{\alpha}([\gamma,\boldsymbol{y}])
    \end{equation*}
    for any $[\gamma,\boldsymbol{y}]\in CW(T^*\Sigma,\mathbf{q},\alpha,\mathbf{p}_0)$ and $\boldsymbol{x}\in CW(T^*_{\mathbf{q}}\Sigma)$. 
    We analyze the boundary of the 1-dimensional moduli space $\overline{\mathcal{H}}^\chi(\boldsymbol{y},\alpha,\mathbf{q},\boldsymbol{x})$ for varying Euler characteristic $\chi$.
    
    As shown in Figure \ref{fig-T2-deg} and Figure \ref{fig-T2-nodal}, there are five types of boundary degenerations:
    \begin{enumerate}
        \item $\coprod_{\boldsymbol{y}',\chi'+\chi''-\kappa=\chi} \mathcal{M}^{\chi'}(\boldsymbol{y}',\boldsymbol{x},\boldsymbol{y}) \times\mathcal{H}^{\chi''}(\boldsymbol{y}',\alpha,\mathbf{q})$;
        \item $\coprod_{\mathbf{q}',\chi'+\chi''-\kappa=\chi}\mathcal{H}^{\chi'}(\boldsymbol{y},\alpha,\mathbf{q}')\times\mathcal{H}^{\chi''}(\boldsymbol{x},\mathbf{q}',\mathbf{q})$;
        \item the set $\partial_\Sigma \overline{\mathcal{H}}^\chi(\boldsymbol{y},\alpha,\mathbf{q},\boldsymbol{x})$ with a nodal degeneration along $\Sigma$;
        \item the set $\partial_\alpha \overline{\mathcal{H}}^\chi(\boldsymbol{y},\alpha,\mathbf{q},\boldsymbol{x})$ with a nodal degeneration along $T^*_\alpha\Sigma$;
        \item the set $\partial_c \overline{\mathcal{H}}^\chi(\boldsymbol{y},\alpha,\mathbf{q},\boldsymbol{x})$ with a nodal degeneration along $\alpha$.
    \end{enumerate}
    (1) is given on the left-hand side of Figure \ref{fig-T2-deg} and contributes to $\mathcal{F}_{\alpha}(\mu_{\alpha}^2(\boldsymbol{x},[\gamma,\boldsymbol{y}]))$; (2) is given on the right-hand side of Figure \ref{fig-T2-deg} and contributes to $\mathcal{F}(\boldsymbol{x})\mathcal{F}_{\alpha}([\gamma,\boldsymbol{y}])$.
    A standard gluing argument shows that all contributions to $\mathcal{F}_{\alpha}(\mu_{\alpha}^2(\boldsymbol{x},[\gamma,\boldsymbol{y}]))$ and $\mathcal{F}(\boldsymbol{x})\mathcal{F}_{\alpha}([\gamma,\boldsymbol{y}])$ come from such broken degenerations. 

    We now discuss (3), which is shown in Case I of Figure \ref{fig-T2-nodal}. 
    Let $v_s\in{\mathcal{H}}^\chi(\boldsymbol{y},\alpha,\mathbf{q},\boldsymbol{x})$, $s\in[-\epsilon,0)$, be a generic 1-parameter family such that $\pi_{D_3}\circ v_s$ has a branched point (generically a double branched point) that limits to the zero section $\Sigma$ as $s\to 0$.
    We denote this limiting curve as $v_0\in\partial_{\Sigma}\overline{\mathcal{H}}^{\chi}(\boldsymbol{y},\alpha,\mathbf{q},\boldsymbol{x})$.
    By Gromov compactness, $\partial_{\Sigma}\overline{\mathcal{H}}^{\chi}(\boldsymbol{y},\alpha,\mathbf{q},\boldsymbol{x})$ consists of finite points.
    If we continue this family past the nodal curve $v_0$, then the nodal point is removed and $\chi$ increases by 1 as seen on the right-hand side of Case I of Figure \ref{fig-T2-nodal}. 
    Interpreted in another way, given $v_0\in \partial_{\Sigma}\overline{\mathcal{H}}^{\chi}(\boldsymbol{y},\alpha,\mathbf{q},\boldsymbol{x})$, there exists a 1-parameter family of curves $u_t$, $t\in[-\epsilon,\epsilon]$, in $\mathcal{H}^{\chi+1}(\boldsymbol{y},\alpha,\mathbf{q},\boldsymbol{x})$ such that $u_0$ is obtained from $v_0$ by removing the nodal point and the family $\mathcal{E}_\Sigma(u_0)$ corresponds to a crossing of two paths in $[0,1]\times \Sigma$.
    This is illustrated in Figure \ref{fig-skein-brane} as a skein relation on $[0,1]\times\Sigma$.
    
    On the other hand, given a 1-parameter family $u_t$, $t\in[-\epsilon,\epsilon]$, of curves in $\mathcal{H}^{\chi+1}(\boldsymbol{y},\alpha,\mathbf{q},\boldsymbol{x})$ such that $\mathcal{E}_\Sigma(u_0)$ exhibits a single crossing of two strands (see the bottom blue line of Figure \ref{fig-skein-brane}), let $v_0$ be the nodal curve corresponding to $u_0$ as in the middle of Case I of Figure \ref{fig-T2-nodal}. 
    Then there exists a unique 1-parameter family of curves $v_s\in{\mathcal{H}}^\chi(\boldsymbol{y},\alpha,\mathbf{q},\boldsymbol{x})$, $s\in[-\epsilon,0)$ which limits to $v_0$ as $s\to0$.
    We refer the reader to \cite[Lemma 4.16]{ekholm2021skeins} for the gluing details.
    Therefore, a small neighborhood of $v_0$ in $\overline{\mathcal{H}}^{\chi}(\boldsymbol{y},\alpha,\mathbf{q},\boldsymbol{x})$ is homeomorphic to $[-\epsilon,0]$, which corresponds to the upper arc of Figure \ref{fig-skein-brane}. 
    By the HOMFLY skein relation in the definition of $\op{BSk}_\kappa(\Sigma,\mathbf{q},\alpha)$, we see that $\mathcal{G}(u_{-\epsilon})-\mathcal{G}(u_{+\epsilon})=\mathcal{G}(v_{-\epsilon})$.

    \vskip0.2cm
    The discussion of (4), as shown in Case II of Figure \ref{fig-T2-nodal}, is exactly the same as (3), where instead of $\op{BSk}_\kappa(\Sigma,\mathbf{q},\alpha)$ we consider the HOMFLY skein relation on $\op{BSk}^\circ_\kappa(T^*_\alpha\Sigma,\mathbf{p}_0,\alpha)$.
    
    \vskip0.2cm
    Finally we discuss (5), which is given by Case III of Figure \ref{fig-T2-nodal}. 
    This is also similar to (3).
    Let $v_s\in{\mathcal{H}}^\chi(\boldsymbol{y},\alpha,\mathbf{q},\boldsymbol{x})$, $s\in[-\epsilon,0)$, be a generic 1-parameter family such that $\pi_{D_3}\circ v_s$ has a branched point that limits to the intersection $\alpha=\Sigma\cap T^*_\alpha\Sigma$ as $s\to 0$.
    We denote this limiting curve as $v_0\in\partial_{c}\overline{\mathcal{H}}^{\chi}(\boldsymbol{y},\alpha,\mathbf{q},\boldsymbol{x})$.
    There exists a 1-parameter family of curves $u_t$, $t\in[-\epsilon,\epsilon]$, in $\mathcal{H}^{\chi+1}(\boldsymbol{y},\alpha,\mathbf{q},\boldsymbol{x})$ such that $u_0$ is obtained from $v_0$ by removing the nodal point as seen on the right-hand side of Case III of Figure \ref{fig-T2-nodal}. 
    This is shown in the upper case of Figure \ref{fig-corner-skein}.
    
    On the other hand, given a 1-parameter family $u_t$, $t\in[-\epsilon,\epsilon]$, of curves in $\mathcal{H}^{\chi+1}(\boldsymbol{y},\alpha,\mathbf{q},\boldsymbol{x})$ such that $\mathcal{E}_\Sigma(u_0)$ and $\mathcal{E}_\alpha(u_0)$ both exhibit a single crossing of two strands at $\alpha$, let $v_0$ be the nodal curve corresponding to $u_0$ as in the middle of Case III of Figure \ref{fig-T2-nodal}. 
    Then there exists a unique 1-parameter family of curves $v_s\in{\mathcal{H}}^\chi(\boldsymbol{y},\alpha,\mathbf{q},\boldsymbol{x})$, $s\in[-\epsilon,0)$ which limits to $v_0$ as $s\to0$.
    This is shown by the local model in Section \ref{section-local}.
    By the skein relation at a corner (\ref{eq-corner}) in the definition of $\op{BSk}_\kappa(\Sigma,\mathbf{q},\alpha)$ and $\op{BSk}^\circ_\kappa(T^*_\alpha\Sigma,\mathbf{p}_0,\alpha)$, we see that $\mathcal{G}(u_{-\epsilon})-\mathcal{G}(u_{+\epsilon})=\mathcal{G}(v_{-\epsilon})$.

    \vskip0.2cm
    By comparing the definition of $\mathcal{F}_{\alpha}$ and $\mathcal{G}$, we can conclude that $\mathcal{F}_{\alpha}(\mu^{2}_{\alpha}(\boldsymbol{x},[\gamma,\boldsymbol{y}])=\mathcal{F}(\boldsymbol{x})\mathcal{F}_{\alpha}([\gamma,\boldsymbol{y}])$.
\end{proof}


\vskip0.1in

We now evaluate skeins in $\mathrm{BSk}_\kappa(T^*_\alpha\Sigma,\mathbf{p}_0)$ to scalars in $\mathbb{Z}[s,s^{-1}]$. 
The braid skein algebra $\mathrm{BSk}_\kappa(T^*_\alpha\Sigma,\mathbf{p}_0)$ of an annulus is isomorphic to the affine Hecke algebra $\dot{H}_{\kappa}$ \cite{honda2022jems}.
Recall that $\dot{H}_{\kappa}$ is generated by $X_1, \dots, X_{\kappa}, T_1, \dots , T_{\kappa-1}$ such that $T_i$'s generate the finite Hecke algebra, and $T_iX_iT_i=T_{i+1}$. 
Consider the 1-dimensional representation of the affine Hecke algebra
\begin{align*}
  \varphi: & \dot{H}_{\kappa}  \to  \mathbb{Z}[s,s^{-1}] \\
  & T_i  \mapsto  s, \\
  & X_i  \mapsto  s^{2i-1-\kappa}. 
\end{align*}

We then define an equivalence relation on $CW(T^*\Sigma,\mathbf{q},\alpha,\mathbf{p}_0)$: 
\begin{equation}
    [\gamma,\boldsymbol{x}]\sim k[\gamma',\boldsymbol{x}] \text{ if and only if }\varphi(\gamma\cdot\gamma'^{-1})=k\in\mathbb{Z}[s,s^{-1}],
\end{equation}
where $\gamma\cdot\gamma'^{-1}\in\mathrm{BSk}_\kappa(T^*_\alpha\Sigma,\mathbf{p}_0)$, and $\gamma'^{-1}$ denotes the inverse path of $\gamma'$.
Denote $CW(T^*\Sigma,\mathbf{q},\alpha,\mathbf{p}_0)^\varphi\coloneqq CW(T^*\Sigma,\mathbf{q},\alpha,\mathbf{p}_0)/{\sim}$.

Similarly, we define an equivalence relation on $\mathrm{BSk}_\kappa(\mathbf{q},\alpha,\mathbf{p}_0)$:
\begin{equation}
    (\gamma,\eta)\sim k(\gamma',\eta)\text{ if and only if }\varphi(\gamma\cdot\gamma'^{-1})=k\in\mathbb{Z}[s,s^{-1}],
\end{equation}
where $\gamma\cdot\gamma'^{-1}\in\mathrm{BSk}_\kappa(T^*_\alpha\Sigma,\mathbf{p}_0)$.
Denote $\op{BSk}_\kappa(\Sigma,\mathbf{q},\alpha,\mathbf{p}_0)^\varphi= \op{BSk}_\kappa(\Sigma,\mathbf{q},\alpha,\mathbf{p}_0)/{\sim}$.

\begin{lemma}
   The map $\mathcal{F}_{\alpha}$ induces a homomorphism:
    \begin{equation*}
        \overline{\mathcal{F}}_{\alpha}\colon HW(T^*\Sigma,\mathbf{q},\alpha,\mathbf{p}_0)^\varphi\to\op{BSk}_\kappa(\Sigma,\mathbf{q},\alpha,\mathbf{p}_0)^\varphi\otimes_{\mathbb{Z}[\hbar]}\mathbb{Z}\llbracket\hbar\rrbracket.
    \end{equation*}
\end{lemma}
\begin{proof}
    Given $[\gamma,\boldsymbol{x}]\sim[\gamma',\boldsymbol{x}]\in CW(T^*\Sigma,\mathbf{q},\alpha,\mathbf{p}_0)$, by definition,
    \begin{align*}
        \mathcal{F}_{\alpha}([\gamma,x])=\,&\sum_{u\in\mathcal{H}(\boldsymbol{x},\alpha,\mathbf{q})} (-1)^{\natural(u)}\cdot c^{2\langle u,\star\rangle}\cdot\hbar^{\kappa-\chi(u)}\cdot [(\gamma\cdot\mathcal{E}_\alpha(u),\mathcal{E}_\Sigma(u))]\\
        =\,&\sum_{u\in\mathcal{H}(\boldsymbol{x},\alpha,\mathbf{q})} (-1)^{\natural(u)}\cdot c^{2\langle u,\star\rangle}\cdot\hbar^{\kappa-\chi(u)}\cdot [(\gamma'\cdot\mathcal{E}_\alpha(u),\mathcal{E}_\Sigma(u))]\\
        =\,&\mathcal{F}([\gamma',x])
    \end{align*}
    in $\op{BSk}_\kappa(\Sigma,\mathbf{q},\alpha,\mathbf{p}_0)^\varphi\otimes_{\mathbb{Z}[\hbar]}\mathbb{Z}\llbracket\hbar\rrbracket$. Here the second equation is due to 
    \begin{equation*}
        \varphi\left((\gamma\cdot\mathcal{E}_\alpha(u))\cdot(\gamma'\cdot\mathcal{E}_\alpha(u))^{-1}\right)=\varphi(\gamma\cdot\mathcal{E}_\alpha(u)\cdot{\mathcal{E}_\alpha(u)}^{-1}\cdot\gamma'^{-1})=\varphi(\gamma\cdot\gamma'^{-1})=1.
    \end{equation*}
    Therefore, $\mathcal{F}_{\alpha}$ induces a map from $HW(T^*\Sigma,\mathbf{q},\alpha,\mathbf{p}_0)^\varphi$ to $\op{BSk}_\kappa(\Sigma,\mathbf{q},\alpha,\mathbf{p}_0)^\varphi\otimes_{\mathbb{Z}[\hbar]}\mathbb{Z}\llbracket\hbar\rrbracket$, denoted by $\overline{\mathcal{F}}_{\alpha}$.
\end{proof}

For each generator $\boldsymbol{x}\in CW(T_{\mathbf{q}}^*\Sigma,T_{\alpha}^*\Sigma)$, we fix $\gamma_{\boldsymbol{x}}\in \Omega(\op{UConf}_\kappa(T^*_\alpha\Sigma),\mathbf{p}_0,\boldsymbol{x}(1))$;
Fix an element $\gamma_0\in\Omega(\op{UConf}_\kappa(T^*_\alpha\Sigma),\mathbf{p}_0,\alpha)$.
For each $[\eta]\in\mathrm{BSk}_\kappa(\Sigma,\mathbf{q},\alpha)$, we pick $\eta$ such that $\eta(1)=\gamma_0(1)$.
Define the quotient maps
\begin{gather}
    \Psi\colon CW(T^*\Sigma,\mathbf{q},\alpha,\mathbf{p}_0)^\varphi\to CW(T_{\mathbf{q}}^*\Sigma,T_{\alpha}^*\Sigma),\quad [\gamma_{\boldsymbol{x}},\boldsymbol{x}]\mapsto\boldsymbol{x},\\
    \Phi\colon\op{BSk}_\kappa(\Sigma,\mathbf{q},\alpha,\mathbf{p}_0)^\varphi\to\mathrm{BSk}_\kappa(\Sigma,\mathbf{q},\alpha),\quad (\gamma_0,\eta)\mapsto \eta.\label{eq-Phi}
\end{gather}

\begin{lemma}
\label{lemma-F-iso-0}
    The restriction of $\overline{\mathcal{F}}_{\alpha}$ to $\hbar=0$ is an isomorphism.
\end{lemma}
\begin{proof}
    First note that when $\hbar=0$, $\varphi(\gamma)=1$ for all $\gamma\in\Omega(\mathrm{UConf}_{\kappa}(T^*_\alpha\Sigma),\mathbf{p}_0)$.
    Therefore, $\Psi$ and $\Phi$ are isomorphisms by definition.

    Consider the case of $\kappa=1$. 
    By Theorem 7.1 of \cite{abbondandolo2010floer}, the map $\Phi\circ\overline{\mathcal{F}}_{\alpha}\circ\Psi^{-1}\colon CW(T_{\mathbf{q}}^*\Sigma,T_{\alpha}^*\Sigma)\to\mathrm{BSk}_\kappa(\Sigma,\mathbf{q},\alpha)$ is an isomorphism, hence $\overline{\mathcal{F}}_{\alpha}$ is an isomorphism.

    For general $\kappa\geq1$, the map $\Phi\circ\overline{\mathcal{F}}_{\alpha}\circ\Psi^{-1}$ is also well-defined and an isomorphism by the same arguments as \cite[Lemma 6.6]{honda2022jems}.

    Therefore, $\overline{\mathcal{F}}_{\alpha}|_{\hbar=0}$ is an isomorphism.
\end{proof}

\begin{lemma}
\label{lemma-psi-iso}
    $\op{BSk}_\kappa(\Sigma,\mathbf{q},\alpha,\mathbf{p}_0)^\varphi$ is isomorphic to $\mathrm{BSk}_\kappa(\Sigma,\mathbf{q},\alpha)$.
\end{lemma}
\begin{proof}
    Fix $\gamma_0\in\Omega(\op{UConf}_\kappa(T^*_\alpha\Sigma),\mathbf{p}_0,\alpha)$ as in (\ref{eq-Phi}).
    For each $[\eta]\in\mathrm{BSk}_\kappa(\Sigma,\mathbf{q},\alpha)$, we pick $\eta$ such that $\eta(1)=\gamma_0(1)$.
    We then define 
    \begin{gather*}
        \rho\colon\mathrm{BSk}_\kappa(\Sigma,\mathbf{q},\alpha)\to\op{BSk}_\kappa(\Sigma,\mathbf{q},\alpha,\mathbf{p}_0)^\varphi,\quad [\eta]\mapsto [(\gamma_0,\eta)].
    \end{gather*}
    We leave it to the reader to check that $\rho$ is well-defined and is an isomorphism.
\end{proof}

\noindent {\em Proof of Theorem \ref{thm main1}:}
    It follows from Lemma \ref{lemma-F-iso-0}, Lemma \ref{lemma-psi-iso}, and the same proof of \cite[Theorem 1.4]{honda2022jems}.
\qed

\section{The polynomial representation of $\ddot{H}_{\kappa}$}
We focus on $\Sigma=T^2$ in this section. 
We show that the braid skein module is isomorphic to the polynomial representation $P_{\kappa}$ of $\ddot{H}_{\kappa}$; see Propostion \ref{prop BSk poly}. 
Combining with Theorem \ref{thm main1}, it gives a HDHF realization of the polynomial representation.  
We then give a topological interpretation of Cherednik's inner product on the polynomial representation $P_{\kappa}$.

\subsection{Interpretation via skeins and HDHF}

Consider the 2-torus $T^2=S^1_x\times S^1_y$ where $S^1=[0,1]/\{0\sim 1\}$ equipped with the standard Riemannian metric. 
Let $\alpha=\{0\}_x\times S^1_y$.
Let $\mathbf{q}=\{q_1,\dots,q_\kappa\}\subset T^2$ be a $\kappa$-tuple of basepoints on $\{0\}_x\times (0,1)_y$ with an increasing order of $y$-coordinates. 
We further fix the marked point $\star$ with its $y$-coordinate as $0$. 

The double affine Hecke algebra $\ddot{H}_{\kappa}$ is generated by
\begin{equation*}
    \{X_1,\dots,X_\kappa,Y_1,\dots,Y_\kappa,T_1,\dots,T_{\kappa-1}\}.
\end{equation*}
According to \cite{morton2021dahas}, $\ddot{H}_{\kappa}$ is isomorphic to $\mathrm{BSk}_\kappa(T^2,\mathbf{q})$ by the identification in Figure \ref{fig-DAHA}.
The variables $s, c$ in $\mathrm{BSk}_\kappa(T^2,\mathbf{q})$ matches $t, q$ in $\ddot{H}_{\kappa}$ via $s^2 \to t, c^2 \to q^{-1}$.

\begin{figure}[ht]
    \centering
    \includegraphics[width=14cm]{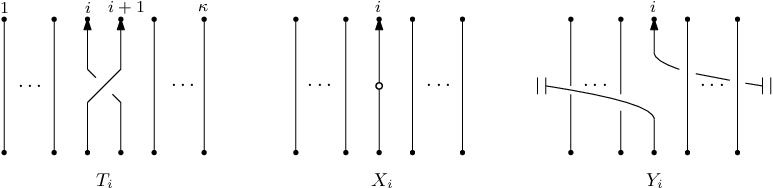}
    \caption{The circle in the $i$-th strand denotes the degree 1 loop in the $S^1_x$ direction.}
    \label{fig-DAHA}
\end{figure}

Recall that the polynomial representation $P_{\kappa}$ of $\ddot{H}_{\kappa}$ is an induced representation. 
More precisely, $P_{\kappa}$ is isomorphic to a quotient of $\ddot{H}_{\kappa}$ by the following relations: 
\begin{itemize}
    \item (A1): $\Theta T_i \sim s\Theta$, $i=1,\dots,\kappa-1$,
    \item (A2): $\Theta Y_i \sim s^{\kappa+1-2i}\Theta$, $i=1,\dots,\kappa$,
\end{itemize}
for $\Theta\in \ddot{H}_{\kappa}$.
Denote $X^{\mathbf{a}}=X_1^{a_1}\cdots X_{\kappa}^{a_{\kappa}}$ and $Y^{\mathbf{b}}=Y_1^{b_1}\cdots Y_{\kappa}^{b_{\kappa}}$ for $\mathbf{a}=(a_1,\dots, a_{\kappa}), \mathbf{b}=(b_1,\dots, b_{\kappa})\in \Z^{\kappa}$.
Since $\ddot{H}_{\kappa}$ has a PBW basis $\{X^\mathbf{a}Y^\mathbf{b}T_w,\,\mathbf{a},\mathbf{b}\in\mathbb{Z}^\kappa,w\in S_\kappa\}$, it follows that this quotient has a basis $\{X^{\mathbf{a}},\mathbf{a}\in\mathbb{Z}^\kappa\}$ which is isomorphic to $P_{\kappa}$.

\begin{proposition}
\label{prop BSk poly}
    There is an isomorphism
    \begin{equation*}
        P_\kappa\to\mathrm{BSk}_\kappa(T^2,\mathbf{q},\alpha),
    \end{equation*}
    which commutes with the actions of $\ddot{H}_{\kappa}$ and $\op{BSk}_\kappa(\Sigma,\mathbf{q})$, respectively.
\end{proposition}
\begin{proof}
We first define a map
        $f:P_{\kappa}\to \mathrm{BSk}_\kappa(T^2,\mathbf{q},\alpha)$ as follows.
    Given a representative $\Theta \in \ddot{H}_{\kappa}$ of $[\Theta]\in P_{\kappa}$, it corresponds to a braid in $\mathrm{Br}_{\kappa,1}(T^2,\mathbf{q},\star) \cong  \ddot{H}_{\kappa}$. 
    This braid can also be viewed as a braid in
    $\mathrm{Br}_{\kappa,1}(T^2,\mathbf{q},\alpha,\star)$ since we have chosen $\mathbf{q}\subset\alpha$. Let $f(\Theta)$ be its class in $\mathrm{BSk}_\kappa(T^2,\mathbf{q},\alpha)$. 
    It is easy to check that $f(\Theta)=f(\Theta')$ if $[\Theta]=[\Theta']\in P_{\kappa}$, since the relations (A1), (A2) are incorporated in (\ref{eq-corner}).
    
    \begin{figure}[ht]
        \centering
        \includegraphics[width=14cm]{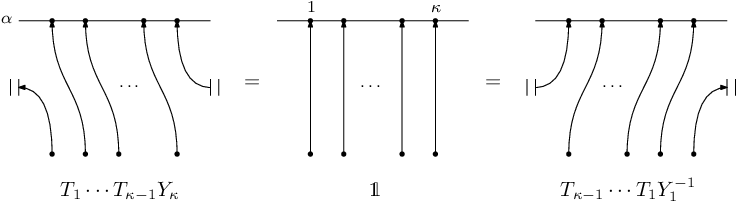}
        \caption{}
        \label{fig-alpha-relation}
    \end{figure}

    Now we define a map
    $g:\mathrm{BSk}_\kappa(T^2,\mathbf{q},\alpha)\to P_{\kappa}$ in the inverse direction.
    Given a braid representative $\gamma=(\gamma_1,\dots,\gamma_\kappa)\in\mathrm{Br}_{\kappa,1}(T^2,\mathbf{q},\alpha,\star)$ of $[\gamma]\in \mathrm{BSk}_\kappa(T^2,\mathbf{q},\alpha)$, we order $(\gamma_1,\dots,\gamma_\kappa)$ such that $0<\gamma_{i_1}(1)<\dots<\gamma_{i_\kappa}(1)<1$ in the $y$-coordinates.
    We apply an isotopy within $\{0\}_x\times (0,1)_y$ such that $(\gamma_{i_1}(1),\dots,\gamma_{i_\kappa}(1))=(q_1,\dots,q_\kappa)$. This gives a braid in $\mathrm{Br}_{\kappa,1}(T^2,\mathbf{q},\star) \cong \ddot{H}_{\kappa}$. We define $g(\gamma) \in P_{\kappa}$ as the class of this braid under the quotient map $\ddot{H}_{\kappa} \to P_{\kappa}$.
    
    We need to check that $g(\gamma)$ does not depend on the choice of $\gamma$ in $[\gamma]\in\mathrm{BSk}_\kappa(T^2,\mathbf{q},\alpha)$.
    It suffices to check the invariance under the isotopy of $\gamma$ in $\mathrm{Br}_{\kappa,1}(T^2,\mathbf{q},\alpha,\star)$ shown in Figure \ref{fig-alpha-relation}.
    This follows from
    \begin{equation*}
        [\Theta T_1\cdots T_{\kappa-1}Y_\kappa]=[\Theta T_{\kappa-1}\cdots T_1 Y_1^{-1}]=[\Theta] \in P_{\kappa},
    \end{equation*}
    for $\Theta\in \ddot{H}_{\kappa}$, due to the relations (A1) and (A2). 

    It is straightforward to check that $f$ and $g$ are inverses to each other. 
    This isomorphism is apparently commuting with the actions since both of them are the corresponding induced representations. 
\end{proof}

    

Combining with Theorem \ref{thm main1}, it gives a HDHF realization of the polynomial representation, i.e. Corollary \ref{cor}.  


\subsection{Cherednik's inner product}
We change the coefficient ring $\Z[s^{\pm 1},c^{\pm 1}]$ to the fractional field $\Q(s,c)$ throughout this subsection. 
Consider the path space
\begin{equation*}
    \Omega(\op{UConf}_\kappa(\Sigma),\alpha,\alpha)\coloneqq \{\gamma\in C^0([0,1],\op{UConf}_\kappa(\Sigma))\,|\,\gamma(0)\subset\alpha,\gamma(1)\subset \alpha\}.
\end{equation*}

Denote the set of homotopy classes of paths in $\Omega(\op{UConf}_\kappa(\Sigma),\alpha,\alpha)$, viewed as braids that start from $\alpha$ and end on $\alpha$, by $\op{Br}_\kappa(\Sigma,\alpha,\alpha)$.
Similarly, we define $\op{Br}_{\kappa,1}(\Sigma,\alpha,\alpha,\star)$ to consist of braids that start from $\alpha\sqcup\{\star\}$ and end on $\alpha\sqcup\{\star\}$ such that the last strand connects $\star$ to itself by a straight line in $[0,1]\times\Sigma$.

\begin{definition}
    \label{def-skein-alpha}
    The braid skein module $\mathrm{BSk}_\kappa(\Sigma,\alpha,\alpha)$ is the quotient of the vector space $\Q(s,c)[\mathrm{Br}_{\kappa,1}(\Sigma,\alpha,\alpha,\star)]$ by the same local relations in Definition \ref{def-skein-module}. 
\end{definition}

Consider the composition:
$$\op{BSk}_\kappa(T^2,\mathbf{q},\alpha) \otimes \op{BSk}_\kappa(T^2,\mathbf{q},\alpha) \to \op{BSk}_\kappa(T^2,\alpha,\mathbf{q}) \otimes \op{BSk}_\kappa(T^2,\mathbf{q},\alpha) \to \op{BSk}_\kappa(T^2,\alpha,\alpha),$$
where the first map is taking the inverse braid for the first factor, and the second map is a concatenation of braids along $\mathbf{q}$. 
We will show that the image of the map is one dimensional so that we have a bilinear form on $\op{BSk}_\kappa(T^2,\mathbf{q},\alpha)$.

Given $[\gamma]\in\mathrm{BSk}_\kappa(\Sigma,\alpha,\alpha)$, where $\gamma=\{\gamma_1,\dots,\gamma_\kappa\} \in \Omega(\op{UConf}_\kappa(\Sigma),\alpha,\alpha)$ (the constant strand at $\star$ is omitted), define the degree of $[\gamma]$ as an integer
\begin{equation*}
    \theta(\gamma_1)+\dots+\theta(\gamma_\kappa),
\end{equation*}
where $\theta$ is an identification of $H_1(\Sigma,\alpha;\mathbb{{Z}})$ with $\mathbb{Z}$. 
The degree of $[\gamma]$ is independent of choices of $\gamma$.
Let $\mathrm{BSk}^0_\kappa(\Sigma,\alpha,\alpha)$ be the degree $0$ summand of $\mathrm{BSk}_\kappa(\Sigma,\alpha,\alpha)$.
We denote the constant braid in $\mathrm{BSk}^0_\kappa(\Sigma,\alpha,\alpha)$ as $[\gamma_c]$. Note that the class $[\gamma_c]$ does not depend on the positions of endpoints of $\gamma_c$.

We compute $\mathrm{BSk}^0_\kappa(\Sigma,\alpha,\alpha)$ in a similar way as 
computing $\mathrm{BSk}_\kappa(\Sigma,\mathbf{q},\alpha)$ in Proposition \ref{prop BSk poly}.  
Define $Q_{\kappa}$ as a quotient of $\ddot{H}_{\kappa} \otimes_{\Z[s^{\pm 1}, c^{\pm 1}]} \mathbb{Q}(s,c)$ by the relations (A1), (A2), and the following relations: 
\begin{itemize}
    \item (B1): $T_i\Theta  \sim s\Theta$, $i=1,\dots,\kappa-1$,
    \item (B2): $Y_i\Theta  \sim s^{\kappa+1-2i}\Theta$, $i=1,\dots,\kappa$,
\end{itemize}
for $\Theta\in \ddot{H}_{\kappa}$. 
A proof similar to that of Proposition \ref{prop BSk poly} implies that 
\begin{equation*}
    \mathrm{BSk}_\kappa(\Sigma,\alpha,\alpha) \cong Q_{\kappa}.
\end{equation*}

Let $Q_{\kappa}^0$ be the subspace of $Q_{\kappa}$ corresponding to $\mathrm{BSk}^0_\kappa(\Sigma,\alpha,\alpha)$. 
It is known that the vector space $Q^0_{\kappa}$ is of dimension $1$. 
Nevertheless, we give an elementary proof for type A in the following.

We first give an equivalent description of $Q_{\kappa}$. 
Let $g=T_1\cdots T_{\kappa-1}Y_\kappa$. Then $Q_{\kappa}$ is also a quotient of $\ddot{H}_{\kappa}$ by the relations: 
\begin{itemize}
    \item (A): $\Theta T_i  \sim s\Theta$, and $\Theta g \sim \Theta$,
    \item (B): $T_i\Theta  \sim s\Theta$, and $g^{-1}\Theta g  \sim \Theta$,
\end{itemize}
for $\Theta\in \ddot{H}_{\kappa}$. 
Since $\ddot{H}_{\kappa}$ has a PBW basis, we can further restrict  $\Theta$ to be $X^{\mathbf{a}}$ for $\mathbf{a}=(a_1,\dots, a_{\kappa}) \in \Z^{\kappa}$ in the relations above. 
We have $g^{-1} X^{\mathbf{a}} g =c^{-2a_1}X^{\eta(\mathbf{a})}$, where $\eta(\mathbf{a})=(a_2,\dots,a_{\kappa},a_1)$ is the cyclic shift of $\mathbf{a}$; see \cite[Proposition 5.4]{guo2022comparing}.

Let $\dot{H}_{\kappa}$ denote the affine Hecke algebra generated by $X_i, T_i$'s. So $Q_{\kappa}$ can be written as a quotient of $\dot{H}_{\kappa}$ by the following relations:
\begin{itemize}
    \item (R1): $T_iX^{\mathbf{a}} \sim sX^{\mathbf{a}}$, and $X^{\mathbf{a}}T_i \sim sX^{\mathbf{a}}$,
    \item (R2): $X^{\mathbf{a}} \sim c^{-2a_1}X^{\eta(\mathbf{a})}$.
\end{itemize}
where $X^{\mathbf{a}}=X_1^{a_1}\cdots X_{\kappa}^{a_{\kappa}}$ for $\mathbf{a}=(a_1,\dots, a_{\kappa}) \in \Z^{\kappa}$.
So $Q_{\kappa}^0\cong \dot{H}_{\kappa}^0 \otimes_{\Z[s^{\pm 1}, c^{\pm 1}]} \mathbb{Q}(s,c) / \sim$, where $\dot{H}_{\kappa}^0$ is the degree zero summand of $\dot{H}_{\kappa}$. 

\begin{lemma}\label{lemma:automatictotalzero}
    The inclusion $\ddot{H}_{\kappa}^0 \subset \ddot{H}_{\kappa}$ after tensoring with $\Q(s,c)$ induces an isomorphism $Q_{\kappa}^0 \cong Q_{\kappa}$.
    Therefore, $\mathrm{BSk}_{\kappa}^0(\Sigma,\alpha,\alpha) \cong \mathrm{BSk}_{\kappa}(\Sigma,\alpha,\alpha)$.
\end{lemma}
\begin{proof}
    Apply the relation (R2) to a tuple $\mathbf{a}=(1,0,\ldots,0)\in\Z^{\kappa}$ $\kappa$-times,
    we get $X_{1} \sim c^{-2\kappa}X_{1}$.
    Since $(1-c^{-2\kappa})$ is invertible in $\mathbb{Q}(s,c)$,
    this implies that $X_{1} \sim 0$ in $Q_{\kappa}$.
    A similar argument by induction shows that if a monomial $X^{\mathbf{a}}$ has non-zero total degree $a_{1}+\dots+a_{\kappa}$,
    then $X^{\mathbf{a}} \sim 0$.
\end{proof}

\begin{remark}
The relations (R1, R2) on the topological side are easy to visualize. 
The first relation is from Equation (\ref{eq-corner}), and the second is from Equation (\ref{eq-c}).
\end{remark}


\begin{lemma}
The dimension of the $\Q(s,c)$-vector space $Q^0_{\kappa}$ is at most $1$.
\end{lemma}
\begin{proof}
The affine Hecke algebra $\dot{H}_{\kappa}$ has a PBW basis $\{T_wX^{\mathbf{a}}; w \in S_{\kappa}, \mathbf{a} \in \Z^{\kappa}\}$. 
Since $[T_wX^{\mathbf{a}}]=s^{l(w)}[X^{\mathbf{a}}]$, where $l(w)$ is the length of $w$, it is enough to show that $[X^{\mathbf{a}}]$ is proportional to $[1]$ for any $\mathbf{a}$ with $\sum_{i}a_i=0$.  

Let $W$ be the subspace of $\dot{H}^0_{\kappa} / \sim$ spanned by $[1]$. We show that $[X^{\mathbf{a}}] \in W$ for any $\mathbf{a}$ by induction on a complexity index $||\mathbf{a}||=(\sum_i|a_i|, \alpha(\mathbf{a}), \beta(\mathbf{a}))$. 
Here, $\alpha(\mathbf{a})=\max\{a_i\}$, and $\beta(\mathbf{a})=|\{i; a_i=\alpha(\mathbf{a})\}|$. By the cyclic relation (R2), we can assume that $a_1=\alpha(\mathbf{a})$, i.e. $a_1 \geq a_i$ for any $i$. Consider $T_1X^{\mathbf{a}}=T_1X_1^{a_1}X_2^{a_2}\cdots X_{\kappa}^{a_{\kappa}} \in \dot{H}_{\kappa}$. 
If $a_2<0$, then $T_1X_1^{a_1}X_2^{a_2}=X_1^{a_2}X_2^{a_1}T_1^{-1}+X^{\mathbf{b}}$, where $\sum_i|b_i|<\sum_i|a_i|$. 
If $0 \leq a_2 < a_1$, then $T_1X_1^{a_1}X_2^{a_2}=X_1^{a_2}X_2^{a_1}T_1^{-1}+X^{\mathbf{b}}$, where either $\alpha(\mathbf{b})<\alpha(\mathbf{a})$, or $\alpha(\mathbf{b})=\alpha(\mathbf{a}), \beta(\mathbf{b})<\beta(\mathbf{a})$. 
So we have $s[X_1^{a_1}X_2^{a_2}\cdots X_{\kappa}^{a_{\kappa}}]-s^{-1}[X_1^{a_2}X_2^{a_1}\cdots X_{\kappa}^{a_{\kappa}}] \in W$ by induction. 
If $a_2=a_1$, then $T_1X_1^{a_1}X_2^{a_2}=X_1^{a_2}X_2^{a_1}T_1$ so that $[X_1^{a_1}X_2^{a_2}\cdots X_{\kappa}^{a_{\kappa}}]=[X_1^{a_2}X_2^{a_1}\cdots X_{\kappa}^{a_{\kappa}}]$.
Then consider $T_2X_1^{a_2}X_2^{a_1}\cdots X_{\kappa}^{a_{\kappa}}$ and so on. We finally have $s^{r}[X_1^{a_1}X_2^{a_2}\cdots X_{\kappa}^{a_{\kappa}}]-s^{-r}[X_1^{a_2}X_2^{a_3}\cdots X_{\kappa}^{a_1}] \in W$, for some $r$, i.e. 
$$s^{r}[X^{\mathbf{a}}]-s^{-r}[X^{\eta(\mathbf{a})}] \in W.$$
On the other hand, $[X^{\mathbf{a}}]=c^{-2a_1}[X^{\eta(\mathbf{a})}]$ by the relation (R2). We conclude that $[X^{\mathbf{a}}] \in W$ for any $\mathbf{a}$ so that $W=\dot{H}^0_{\kappa} / \sim$.
\end{proof}

\begin{lemma}
The $\Q(s,c)$ vector space $Q^0_{\kappa}$ is nonzero.
\end{lemma}
\begin{proof}
Define a linear map $F: \dot{H}_{\kappa} \to \Q(s,c)$ by $X^{\mathbf{a}}T_w \mapsto s^{l(w)}\langle X^{\mathbf{a}},1\rangle$. 
Here, $\langle-,-\rangle$ is Cherednik's inner product on $P_{\kappa}$.
The map is clearly nonzero.  
It is enough to show that the restriction of $F$ on $\dot{H}^0_{\kappa}$ factors through $\dot{H}^0_{\kappa} / \sim$.
The polynomial representation $P_{\kappa}$ is the induced module of $\dot{H}_{\kappa}$. Let $G: \dot{H}_{\kappa} \to P_{\kappa}$ denote the quotient map. It follows from $G(X^{\mathbf{a}}T_w)=s^{l(w)}X^{\mathbf{a}}$ that $F$ factors through $G$:
$$F: \dot{H}_{\kappa} \xrightarrow{G} P_{\kappa} \xrightarrow{F'} \Q(s,c).$$
Here, $F'(X^{\mathbf{a}})=\langle X^{\mathbf{a}},1\rangle$.
Then we have 
$$F(T_wX^{\mathbf{a}})=F'(T_w\cdot X^{\mathbf{a}})=\langle T_w\cdot X^{\mathbf{a}},1\rangle=\langle X^{\mathbf{a}},T_w^{-1}\cdot1\rangle=s^{l(w)}F(X^{\mathbf{a}}),$$
where the third equation follows from the adjoint property of Cherednik's inner product.  
We use the adjoint property again for $g \in \ddot{H}_{\kappa}$ in \cite[Proposition 5.4]{guo2022comparing}:
$$F(X^{\mathbf{a}})=\langle X^{\mathbf{a}},g\cdot 1\rangle=\langle g^{-1}\cdot X^{\mathbf{a}},1\rangle=c^{-2a_1}\langle  X^{\eta(\mathbf{a})},1\rangle=c^{-2a_1}F(X^{\eta(\mathbf{a})}).$$
So the map $F$ factors through $\dot{H}^0_{\kappa} / \sim$. 
\end{proof}

Due to the identification of $\mathrm{BSk}^0_\kappa(\Sigma,\alpha,\alpha)$ with $Q^0_{\kappa}$, we have the following.
\begin{proposition}
    \label{prop-dim-alpha}
    The $\Q(s,c)$ vector space $\mathrm{BSk}^0_\kappa(\Sigma,\alpha,\alpha)$ is of dimension $1$, and spanned by the constant braid $[\gamma_c]$.
\end{proposition}

By Proposition \ref{prop-dim-alpha} we can write $[\gamma]/[\gamma_c]\in \Q(s,c)$ for $[\gamma]\in\mathrm{BSk}^0_\kappa(\Sigma,\alpha,\alpha)$.

Given $\gamma \in \Omega(\op{UConf}_\kappa(\Sigma),\mathbf{q},\alpha)$, let $\gamma^{-1} \in \Omega(\op{UConf}_\kappa(\Sigma),\alpha,\mathbf{q})$ denote the inverse path $\gamma^{-1}(t)=\gamma(1-t)$ for $t \in [0,1]$.
By Lemma \ref{lemma:automatictotalzero}, the natural map $\mathrm{BSk}^0_\kappa(\Sigma,\alpha,\alpha)\to\mathrm{BSk}_\kappa(\Sigma,\alpha,\alpha)$ induced by the inclusion is an isomorphism.
Denote $\pi:\mathrm{BSk}_\kappa(\Sigma,\alpha,\alpha)\to\mathrm{BSk}^0_\kappa(\Sigma,\alpha,\alpha)$ the inverse of the above map.

\begin{definition}
Define a $\Q$-bilinear form on $\mathrm{BSk}_\kappa(\Sigma,\mathbf{q},\alpha)$:
\begin{gather*}
    \Phi:\mathrm{BSk}_\kappa(\Sigma,\mathbf{q},\alpha)\times\mathrm{BSk}_\kappa(\Sigma,\mathbf{q},\alpha)\to \Q(s,c),\\
    \Phi([\gamma],[\gamma'])\coloneqq \pi([(\gamma')^{-1}\cdot\gamma])/[\gamma_c],
\end{gather*}
where $(\gamma')^{-1}\cdot\gamma \in \Omega(\op{UConf}_\kappa(\Sigma),\alpha,\alpha)$ is the concatenation of path $(\gamma')^{-1}$ and $\gamma$.
We further require that $\Phi$ is $\star$-bilinear, i.e., $\Q(s,c)$ linear in the first argument, and $\Q(s,c)$ conjugate linear in the second argument.
\end{definition}


Recall that Cherednik's inner product $\langle-,-\rangle$ on $P_{\kappa}$ is uniquely determined by the following properties:
\begin{enumerate}
\item $\star$-bilinear: $\langle rf, g \rangle=r\langle f, g \rangle=\langle f, r^*g \rangle$ for $r \in \Q(s,c)$, and $f,g \in P_{\kappa}$. Here $r^*$ is the conjugate of $r$ replacing $s,c$ by $s^{-1},c^{-1}$.
\item $\star$-unitary: $\langle H\cdot f, g \rangle=\langle f, H^* \cdot g \rangle$ for $H \in \ddot{H}_{\kappa}$, and $f,g \in P_{\kappa}$. Here $T_i^*=T_i^{-1}, X_i^*=X_i^{-1}, Y_i^*=Y_i^{-1}$.
\item normalization: $\langle 1,1 \rangle=1$.
\end{enumerate}
See \cite[Proposition 3.3.2]{Cherednik} for more detail.

\vspace{.2cm}
\noindent 
{\em Proof of Theorem \ref{thm main2}:} 
It suffices to show that the bilinear form $\Phi$ also satisfies the corresponding defining properties. 
We check $\star$-unitary: 
$$\Phi([\beta]\cdot [\gamma], [\gamma'])=\pi([(\gamma')^{-1}\cdot\beta\cdot \gamma])/[\gamma_c]=\Phi( [\gamma], [\beta]^{-1}\cdot[\gamma']).$$
The other two properties are obvious.
\qed



        

\printbibliography

\end{document}